\documentclass[10pt,a4paper]{article}

\usepackage[T1]{fontenc}
\usepackage[utf8]{inputenc}
\usepackage{amsmath,amssymb,amsthm,amsfonts}
\usepackage{a4wide}

\usepackage{graphicx}  
\usepackage{tikz}      
\usetikzlibrary{arrows,automata,shapes}
\usepackage{chngcntr}  
\counterwithin{table}{section}
\counterwithin{figure}{section}

\theoremstyle{plain}
\newtheorem{theorem}{Theorem}[section]
\newtheorem{proposition}[theorem]{Proposition}
\newtheorem{lemma}[theorem]{Lemma}
\newtheorem{corollary}[theorem]{Corollary}
\theoremstyle{definition}
\newtheorem{definition}[theorem]{Definition}
\newtheorem{remark}[theorem]{Remark}
\newtheorem{example}[theorem]{Example}

\theoremstyle{remark}
\numberwithin{equation}{section}

\usepackage[pdfborder={0 0 0},bookmarksopen]{hyperref}
\hypersetup{%
	pdftitle = {Single Jump Processes and Strict Local Martingales},
	pdfauthor = {Martin Herdegen, Sebastian Herrmann},
	pdfkeywords = {},
	pdfpagemode = UseNone,
}

\newcommand{\e}{\mathrm{e}}

\newcommand{\FF}{\mathbb{F}}
\newcommand{\GG}{\mathbb{G}}
\newcommand{\NN}{\mathbb{N}}
\newcommand{\RR}{\mathbb{R}}

\newcommand{\cA}{\mathcal{A}}
\newcommand{\cF}{\mathcal{F}}
\newcommand{\cG}{\mathcal{G}}
\newcommand{\cM}{\mathcal{M}}

\newcommand{\diff}{\mathrm{d}}
\newcommand{\dd}{\,\mathrm{d}}
\newcommand{\loc}{\mathrm{loc}}

\newcommand{\1}{\mathbf{1}}

\newcommand{\rdbrack}{]\!]}

\newcommand*{\EX}[2][]{E^{#1}\left [ #2 \right ]}
\newcommand*{\cEX}[3][]{E^{#1}\left[ #2 \,\middle\vert\, #3 \right]}
\newcommand*{\as}[1]{#1\text{-a.s.}}

\newcommand*{\mart}[3][]{\cM_{#1}^{#2} #3}
\newcommand*{\op}[1]{\cA^{#1}}
\newcommand{\lac}{\stackrel{\loc}{\ll}}
\newcommand{\barG}{\overline{G}}

\begin{document}
\title{%
Single Jump Processes and Strict Local Martingales%
\footnote{We are deeply grateful to Martin Schweizer for a very careful reading of a first draft and various suggestions, which have considerably improved the quality of the paper. We would also like to thank an associate editor and a referee for helpful comments. Moreover, we gratefully acknowledge financial support by the Swiss Finance Institute and by the National Centre of Competence in Research ``Financial Valuation and Risk Management'' (NCCR FINRISK), Project D1 (Mathematical Methods in Financial Risk Management). The NCCR FINRISK is a research instrument of the Swiss National Science Foundation.}
}
\date{}
\author{%
  Martin Herdegen%
  \thanks{%
  ETH Z\"urich, Department of Mathematics, R\"amistrasse 101, CH-8092 Z\"urich, Switzerland, email
  \href{mailto:martin.herdegen@math.ethz.ch}{\nolinkurl{martin.herdegen@math.ethz.ch}}.
  }
  \and
  Sebastian Herrmann%
  \thanks{
  ETH Z\"urich, Department of Mathematics, R\"amistrasse 101, CH-8092 Z\"urich, Switzerland, email
  \href{mailto:sebastian.herrmann@math.ethz.ch}{\nolinkurl{sebastian.herrmann@math.ethz.ch}}.
  }
}
\maketitle

\begin{abstract}
Many results in stochastic analysis and mathematical finance involve local martingales. However, specific examples of \emph{strict} local martingales are rare and analytically often rather unhandy. We study local martingales that follow a given deterministic function up to a random time $\gamma$ at which they jump and stay constant afterwards. The (local) martingale properties of these \emph{single jump local martingales} are characterised in terms of conditions on the input parameters. This classification allows an easy construction of strict local martingales, uniformly integrable martingales that are not in $H^1$, etc. As an application, we provide a construction of a (uniformly integrable) martingale $M$ and a bounded (deterministic) integrand $H$ such that the stochastic integral $H \bullet M$ is a strict local martingale.
\end{abstract}

\vspace{0.5em}

{\small

\vspace{0.25em}

\vspace{0.25em}
}

\section{Introduction}
\label{sec:introduction}

Strict local martingales, i.e., local martingales which are not martingales, play an important role in mathematical finance, e.g., in the context of modelling financial bubbles \cite{LoewensteinWillard2000, CoxHobson2005, Protter2013, Kreher2014} or in arbitrage theory \cite{Kardaras2012, GuasoniRasonyi2015}. Specific examples of strict local martingales are usually rather complicated, the classical example being the inverse Bessel process \cite{DelbaenSchachermayer1995.Bessel}. The aim of this paper is to study a very tractable class of processes and classify their (local) martingale properties. More precisely, we consider \emph{single jump local martingales}, i.e., processes $\mart{G}{F} = (\mart[t]{G}{F})_{t \in [0,\infty]}$ of the form
\begin{align}
\label{eqn:intro:MGF}
\mart[t]{G}{F}
&= F(t)\1_{\lbrace t < \gamma \rbrace } + \op{G} F(\gamma) \1_{\lbrace t \geq \gamma \rbrace },
\end{align}
where the jump time $\gamma$ is a $(0,\infty)$-valued random variable with distribution function $G$ and $F:[0,\infty) \to \RR$ a function that is ``locally absolutely continuous'' with respect to $G$. In words, each path $\mart[\cdot]{G}{F}(\omega)$ follows a deterministic function $F$ up to some random time $\gamma (\omega)$ and stays constant at $\op{G} F(\gamma(\omega))$ from time $\gamma(\omega)$ on. The function $\op{G} F$ is chosen such that $\mart{G}{F}$ becomes a martingale on the right-open interval $[0, t_G)$, where $t_G := \sup\lbrace t \geq 0 : G(t) < 1\rbrace \in (0, \infty]$ denotes the \emph{right endpoint} of the distribution function $G$. All local martingales studied in this article are of the form \eqref{eqn:intro:MGF}.

The two main advantages of single jump local martingales are their flexibility and tractability. They are flexible enough to include examples of processes in well-known martingale spaces. Considered on the closed interval $[0, t_G]$, or equivalently on $[0, \infty]$, $\mart{G}{F}$ can be either of the following: not even a semimartingale; a nonintegrable local martingale; an integrable strict local martingale; a uniformly integrable martingale which does not belong to $H^1$; an $H^1$-martingale (and of course an $H^p$-martingale for $p > 1$). Our main result is a complete characterisation of these five cases in terms of conditions on the two input parameters $G$ and $F$ (cf.~Figure \ref{fig:diagram}).
As for tractability, single jump local martingales are particularly suited for explicit calculations. For instance, we give a general, direct solution to the problem of finding a bounded (deterministic) integrand $H$ and a martingale $M$ such that the stochastic integral $H \bullet M$ is a strict local martingale. Moreover, using only direct arguments, the authors construct in \cite{HerdegenHerrmann2015.RemovalOfDrift} counter-examples to show that neither of the no-arbitrage conditions NA and NUPBR implies the other. Because of their simple structure, these counter-examples also provide more insight into the nature of the underlying result than the more complicated counter-examples already available.

While the distribution function $G$ of the random time $\gamma$ is a natural input parameter, the choice of $F$ as a second input parameter might be less clear. Another natural approach would be to start with a process $S_t := \delta(\gamma) \1_{\lbrace t \geq \gamma \rbrace }$ for a deterministic function $\delta: [0,\infty) \to \RR$. For $\delta = 1$, this is done in the literature on credit risk in the definition of the ``hazard martingale'', see e.g.~\cite[Proposition 2.1]{ElliottJeanblancYor2000}. If $\delta$ is sufficiently integrable, the compensator (or dual predictable projection; cf.~\cite{JacodShiryaev2003}) $S^{\rm p}$ exists and $M:= S-S^{\rm p}$ is a local martingale of the form \eqref{eqn:intro:MGF}. Yet another possibility is to start with a function $H:(0,\infty)\to\RR$ and to express the function $F = (\cA^G)^{-1} H$ in terms of $H$ and $G$ such that the process
\begin{align}
\label{eqn:intro:alternative parametrisation}
F(t)\1_{\lbrace t < \gamma \rbrace } + H(\gamma) \1_{\lbrace t \geq \gamma \rbrace }
\end{align}
is a martingale on $[0, t_G)$. This is the parametrisation used in \cite{ChouMeyer1975} and \cite{Dellacherie1970}; cf.~the next paragraph. There are at least two reasons why we start our parametrisation with the function $F$ instead of the jump size $\delta$ or the function $H$. First, it turns out that $F$ and $G$ are the natural objects to decide whether $\mart{G}{F}$ belongs to a certain (local) martingale space or not. For instance, if $\mart{G}{F}$ is integrable, then $\mart{G}{F}$ being a strict local martingale is equivalent to a nonvanishing limit $\lim_{t \uparrow\uparrow t_G} F(t)(1-G(t))$ (cf.~Lemma \ref{lem:mass}). If in addition $G$ has no point mass at $t_G$, then $M$ being an $H^1$-martingale is equivalent to $F(\cdot-)$ being $\diff G$-integrable (cf.~Lemma \ref{lem:H1 martingale}). Second, a natural generalisation is to allow the function $F$ to be random and to consider the corresponding process in its natural filtration (this is the subject of forthcoming work). Then the process can follow different trajectories prior to the random time $\gamma$, and observing its evolution corresponds to learning the conditional distribution of $\gamma$ over time. However, if one starts with a process $S_t = \delta \1_{\lbrace t \geq \gamma \rbrace }$ for a random variable $\delta$, such a learning effect is much harder to incorporate, because one would have to construct first the desired filtration and then compute the corresponding compensator. If one simply computes the compensator $S^{\rm p}$ (if is exists) in the natural filtration of $S$, then the local martingale $S-S^{\rm p}$ only has a single possible trajectory prior to $\gamma$ and all information is learnt in a single instant at time $\gamma$.

The study of single jump processes dates back to the classical papers by Dellacherie \cite{Dellacherie1970} and Chou and Meyer \cite{ChouMeyer1975}. Dellacherie \cite{Dellacherie1970} (see also Dellacherie and Meyer \cite[Chapter IV, No.~104]{DellacherieMeyer1978}) starts from the smallest filtration $\FF^\gamma$ with respect to which $\gamma$ is a stopping time. Among other things, he obtains a single jump local martingale by computing the compensator of the process $\1_{\lbrace t \geq \gamma \rbrace }$ in this filtration. He also uses single jump processes to give several counter-examples in the general theory of stochastic processes. However, his simplifying assumption that $t_G = \infty$ immediately excludes the possibility of strict local martingales (cf.~Lemma~\ref{lem:martingale}). In the same setting, Chou and Meyer \cite[Proposition~1]{ChouMeyer1975} show that any local $\FF^\gamma$-martingale null at zero is a (true) martingale on $[0,t_G)$ and of the form \eqref{eqn:intro:alternative parametrisation} with
\begin{align}
\label{eqn:intro:F}
F(t)
&= -\frac{1}{1-G(t)} \int_{(0,t]} H(v) \dd G(v),
\end{align}
and that, conversely, every process of this form is a local $\FF^\gamma$-martingale provided that $H$ is ``locally'' $\diff G$-integrable (so that \eqref{eqn:intro:F} is well defined) and $\Delta G(t_G) = 0$. Our Theorem \ref{thm:local martingale} (a) corresponds to the ``converse'' statement and shows that the localising sequence can be chosen to consist of stopping times with respect to the natural filtration of $\mart{G}{F}$. As this filtration is generally smaller than $\FF^\gamma$, we obtain a slightly stronger statement. \cite[Proposition~1]{ChouMeyer1975} also yields that in the case of $t_G < \infty$ and $\Delta G(t_G) > 0$, processes of the form \eqref{eqn:intro:alternative parametrisation} are always uniformly integrable martingales provided that $H$ is $\diff G$-integrable. Our Theorem \ref{thm:local martingale} (c) shows that in this case, the process is even an $H^1$-martingale. Single jump martingales also appear in the modelling of credit risk, see e.g.~\cite{BieleckiRutkowski2002,ElliottJeanblancYor2000,JeanblancRutkowski2000} and Remark \ref{rem:credit risk}. There the jump time models the default time of a financial asset, and single jump martingales are used to describe the \emph{hazard function} of the default time. Note that in credit risk modelling only single jump \emph{(true) martingales} are considered. To the best of our knowledge, our classification of the (local) martingale properties of single jump processes summarised in Figure \ref{fig:diagram} is new.

The remainder of the paper is structured as follows. Section~\ref{sec:analytic preliminaries} contains basic definitions and all analytic results necessary for the classification of single jump local martingales given in Section~\ref{sec:classification}. Section~\ref{sec:counterexample} presents the counter-example in stochastic integration theory mentioned above.

\section{Analytic preliminaries}
\label{sec:analytic preliminaries}

The proof of the classification of the (local) martingale properties of single jump local martingales is split up into a purely analytic and a stochastic part. In this section, we collect all analytic preliminaries. On a first reading, the reader may wish to go only up to Definition \ref{def:ACloc} and then jump directly to the stochastic part in Section \ref{sec:classification}.

We always fix a distribution function $G: \RR \to [0, 1]$ satisfying ${G(0) = 0}$ and ${G(\infty-)}:=\lim_{t \to \infty} G(t) = 1$. Recall that its \emph{right endpoint} is defined by
\begin{align*}
t_G
&:= \sup\lbrace t \geq 0 : G(t) < 1\rbrace \in (0, \infty].
\end{align*}
For notational convenience, set $G(\infty) := 1$. With this in mind, note that $\Delta G(\infty) = 0$, so that ${\Delta G(t_G) > 0}$ implies $t_G < \infty$. Also, $\diff G$ denotes the Lebesgue--Stieltjes measure on $(0,\infty)$ induced by $G$, and $L^1(\diff G)$ is the space of real-valued $\diff G$-integrable functions. Note that a Borel-measurable function ${\phi:(0,\infty)\to\RR}$ is $\diff G$-integrable if and only if it is $\diff G$-integrable on $(0,t_G)$, since $\diff G$ is concentrated on $(0,t_G]$ and a possible point mass at $t_G$ does not affect the integrability. We call $\phi$ \emph{locally $\diff G$-integrable}, abbreviated by $\phi\in L^1_\loc(\diff G)$, if ${\int_{(0,b]} \vert \phi(v) \vert \dd G(v) < \infty}$ for each $b \in (0, t_G)$. Finally, we set $\barG:=1-G$ which is often called the \emph{survival function} of $G$.

\subsection{Locally absolutely continuous functions}
\label{sec:ACloc}

Classically, a Borel-measurable function $F:[0,\infty)\to\RR$ is called absolutely continuous on the interval $(a,b]$ if there is a Lebesgue-integrable function $f:(a,b]\to\RR$ such that $F(t)-F(a) = {\int_a^t f(v) \dd v}$ for all $t \in (a,b]$; in this case, $f$ is unique a.e.~on $(a,b]$ and is called a density of $F$. Replacing the Lebesgue measure by $\diff G$, we say that $F$ is \emph{absolutely continuous with respect to $G$ on $(a,b]$} if there is a $\diff G$-integrable function $f:(a,b]\to\RR$ such that $F(t)-F(a) = \int_{(a,t]} f(v) \dd G(v)$ for all $t\in(a,b]$. Unlike the Lebesgue measure, $\diff G$ may have atoms, so that the precise choice of the integration domain in the previous integral is important. Our choice of a left-open and right-closed interval is natural as it forces an absolutely continuous $F$ to be right-continuous like $G$. Then $F$ itself induces a signed Lebesgue--Stieltjes measure $\diff F$ on $(a,b]$ which is absolutely continuous with respect to $\diff G$ (restricted to $(a,b]$) in the sense of measures, and $f$ is a version of the Radon--Nikodým density $\frac{\diff F}{\diff G}$ on $(a,b]$.

The following is a local version of this concept.

\begin{definition}
\label{def:ACloc}
A Borel-measurable function $F: [0,\infty) \to \RR$ is called \emph{locally absolutely continuous with respect to $G$ on $(0, t_G)$}, abbreviated as $F \lac G$, if $F$ is absolutely continuous with respect to $G$ on $(0, b]$ for all $0 < b < t_G$. In this case, a Borel-measurable function $f: (0,\infty) \to \RR$ is called a \emph{local density of $F$ with respect to $G$} if for all $0 < b < t_G$, $f$ is a version of the Radon--Nikodým density $\frac{\diff F}{\diff G}$ on $(0, b]$.
\end{definition}

The following result is an easy exercise in measure theory.
\begin{lemma}
\label{lem:ACloc:local density}
Let $F \lac G$. Then $F$ is càdlàg and of finite variation on the half-open interval $[0, t_G)$. Moreover, there exists a local density $f$ of $F$ with respect to $G$; it is $\diff G$-a.e.~unique on $(0, t_G)$ and locally $\diff G$-integrable with
\begin{align}
\label{eqn:lem:ACloc:local density}
\int_{(a, b]} f(v) \dd G(v)
&= F(b) - F(a), \quad 0\leq a < b < t_G.
\end{align}
\end{lemma}

A local density $f$ of $F \lac G$ with respect to $G$ is only $\diff G$-a.e.~unique on $(0, t_G)$ (and not on $(0,t_G]$) and may not be $\diff G$-integrable on $(0,t_G)$, so it may not be a classical Radon--Nikodým density. Nevertheless, we often write---in slight abuse of notation---$f = \frac{\diff F}{\diff G}$. This is justified on the one hand by the above lemma and on the other hand by the fact that we never consider $\frac{\diff F}{\diff G}$ outside $(0, t_G)$.

\begin{remark}
\label{rem:ACloc}
If $F\lac G$, then $F$ need not be càdlàg or of finite variation on the \emph{right-closed} interval $(0, t_G]$. Indeed, define $G: \RR \to [0, 1]$ by $G(t) = t \1_{[0, 1)}(t) + \1_{[1, \infty)}(t) $, i.e., $\diff G$ is a uniform distribution on $(0, 1)$ with $t_G = 1$, and let $F:[0, \infty) \to \RR$ be given by $F(t) = \1_{[0, 1)}(t)\sin\frac{1}{1- t}$. Then $F \lac G$ with local density
\begin{align*}
\frac{\diff F}{\diff G} (v)
&= \1_{(0, 1)}(v)\frac{1}{(1-v)^2} \cos \frac{1}{1- v}.
\end{align*}
However, $F$ is neither càdlàg nor of finite variation on $[0, 1]$.
\end{remark}

\subsection{The function $\op{G} F$}
\label{sec:op}
The first result of this section introduces and analyses the function $\op{G}F$ appearing in the definition of the process $\mart{G}{F}$. Its definition is motivated by the idea that $\mart{G}{F}$ should be a martingale on $[0,t_G]$ provided the function $F$ is nice enough. We refer to the discussion after the proof of Lemma \ref{lem:martingale} for more details.
\begin{lemma}
\label{lem:op}
Let $F \lac G$ and define the function ${\op{G} F: (0, \infty) \to \RR}$ by
\begin{align}
\label{eqn:lem:op}
\op{G} F(v)
&= 
\begin{cases}
F(v-) - \frac{\diff F}{\diff G}(v) \barG(v), \quad &v \in (0, t_G),\\
F(t_G-) \1_{\lbrace \Delta G(t_G) > 0\rbrace }, &v \geq t_G, \text{ if } \lim_{t \uparrow \uparrow t_G} F(t) \text{ exists in }\RR, \\
0, &v \geq t_G, \text{ if } \lim_{t \uparrow \uparrow t_G} F(t) \text{ does not exist in }\RR.
\end{cases}
\end{align}
Then $\op{G} F \in L^1_\loc(\diff G)$ and for all $0 \leq a < b < t_G$,
\begin{align}
\label{eqn:lem:op:integral formula}
\int_{(a, b]} \op{G} F(v)\dd G(v)
&= \Big[- F(v) \barG(v)\Big]^b_a.
\end{align}
Thus, 
\begin{align*}
\op{G} F
&= -\frac{\diff (F\barG)}{\diff G} \quad \diff G \text{-a.e.~on } (0, t_G).
\end{align*}
\end{lemma}

\begin{proof}
Note that $\op{G}F$ is well defined by Lemma \ref{lem:ACloc:local density}.
To prove \eqref{eqn:lem:op:integral formula}, fix $0 \leq a < b < t_G$. The function $F(\cdot-)$ is càglàd and therefore bounded on $(a, b]$, the function $\barG$ is trivially bounded on $(a, b]$, and the function $\frac{\diff F}{\diff G}$ is $\diff G$-integrable on $(a, b]$ by \eqref{eqn:lem:ACloc:local density}. Thus, $\op{G} F \in L^1_\loc(\diff G)$. Associativity of Lebesgue--Stieltjes integrals together with an integration by parts gives the result via
\begin{align*}
\int_{(a, b]} \op{G} F(v) \dd G(v) 
&= \int_{(a, b]} F(v-) \dd G(v) - \int_{(a, b]} \frac{\diff F}{\diff G}(v) \barG(v) \dd G(v)\\
&= \int_{(a, b]} F(v-) \dd G(v) + \Big[- F(v)\barG(v)\Big]^b_a - \int_{(a,b]} F(v-) \dd G(v).\qedhere
\end{align*}
\end{proof}

In general, $\op{G} F$ is not $\diff G$-integrable on $(0,t_G]$. The next result lists some equivalent conditions when this is the case and draws an important consequence. 

\begin{lemma}
\label{lem:op:integrability}
Let $F \lac G$. Then the following are equivalent:
\begin{enumerate}
\item $\op{G}F \in L^1(\diff G)$.
\item $\left(\op{G} F \right)^- \in L^1(\diff G)$ and $\limsup_{t \uparrow \uparrow t_G} F(t)\barG(t) > -\infty$.
\item $\left(\op{G} F \right)^+ \in L^1(\diff G)$ and $\liminf_{t \uparrow \uparrow t_G} F(t)\barG(t) < \infty$.
\end{enumerate}
Moreover, each of the above implies that the limit $\lim_{t \uparrow\uparrow t_G} F(t)\barG(t)$ exists in $\RR$ and
\begin{align}
\label{eqn:lem:op:integrability:open}
\int_{(a, t_G)} \op{G} F(v) \dd G(v)
&= F(a)\barG(a) - \lim_{t \uparrow\uparrow t_G} F(t)\barG(t), \quad a \in [0, t_G).
\end{align}
\end{lemma}

\begin{proof}
``(a) $\Rightarrow$ (b), (c)'': If $\op{G}F \in L^1(\diff G)$, then $\left(\op{G} F \right)^\pm \in L^1(\diff G)$, and dominated convergence and \eqref{eqn:lem:op:integral formula} give
\begin{align*}
\int_{(0, t_G)} \op{G} F(v) \dd G(v) 
&= \lim_{t \uparrow \uparrow t_G} \int_{(0, t]} \op{G} F(v) \dd G(v) 
= \lim_{t \uparrow \uparrow t_G} \Big[-F(v)\barG(v)\Big]_0^t \\
&= F(0) - \lim_{t \uparrow \uparrow t_G} F(t)\barG(t).
\end{align*}
This shows that $\lim_{t \uparrow \uparrow t_G} F(t)\barG(t)$ exists, and \eqref{eqn:lem:op:integrability:open} is satisfied first for $a = 0$ and then, by \eqref{eqn:lem:op:integral formula}, for any $a \in (0, t_G)$.

``(b) $\Rightarrow$ (a)'': Since $\left(\op{G} F \right)^- \in L^1(\diff G)$, it suffices to show that $\int_{(0, t_G)} \op{G} F(v) \dd G(v) < \infty$.
Fatou's lemma applied to $(\op{G}F)^+$, dominated convergence for $(\op{G}F)^-$ and \eqref{eqn:lem:op:integral formula} give
\begin{align*}
\int_{(0, t_G)} \op{G} F(v) \dd G(v) 
&\leq \liminf_{t \uparrow \uparrow t_G}\int_{(0, t]} \op{G} F(v) \dd G(v) 
=F(0) - \limsup_{t \uparrow \uparrow t_G} F(t)\barG(t) < \infty.
\end{align*}

``(c) $\Rightarrow$ (a)'': This is analogous to the proof of ``(b) $\Rightarrow$ (a)''.
\end{proof}

The following result provides further characterisations of the $\diff G$-integrability of $\op{G} F$ in the case ${\Delta G(t_G) > 0}$. In particular, it shows that if $\op{G} F$ is $\diff G$-integrable, then the limit in the second line of the definition of $\op{G} F$ in \eqref{eqn:lem:op} exists in $\RR$.

\begin{lemma}
\label{lem:op:integrability:DeltaG>0}
Let $F \lac G$ and suppose that $\Delta G(t_G) > 0$. The following are equivalent:
\begin{enumerate}
\item $\op{G} F \in L^1(\diff G)$.
\item $\frac{\diff F}{\diff G}\barG \in L^1(\diff G)$.
\item $\frac{\diff F}{\diff G} \in L^1(\diff G)$.
\item $F$ is of finite variation on $[0, t_G]$.
\end{enumerate}
Each of the above implies that the limit $\lim_{t \uparrow \uparrow t_G} F(t)$ exists in $\RR$.
\end{lemma}

\begin{proof}
The last statement follows immediately from (d), ``(d) $\Leftrightarrow$ (c)'' is a standard result in analysis, and ``(c) $\Leftrightarrow$ (b)'' follows immediately from the fact that the function $\barG$ is bounded above by $1$ and below by $\barG(t_G-)=\Delta G(t_G) > 0$ on $(0, t_G)$.

For ``(b) $\Leftrightarrow$ (a)'', it suffices to show that the function $F$ is bounded on the compact interval $[0, t_G]$. (Recall that $\Delta G(t_G) > 0$ implies $t_G < \infty$.) Since $F$ is càdlàg on $[0, t_G)$, it is enough to show that the limit $\lim_{t \uparrow\uparrow t_G} F(t)$ exists in $\RR$. Assuming (b), this follows from the equivalence \hbox{``(b) $\Leftrightarrow$ (d)''} and the first part of the proof. Assuming (a), this follows via Lemma \ref{lem:op:integrability} from the fact that the limit $\lim_{t \uparrow \uparrow t_G} F(t)\barG(t)$ exists in $\RR$ and that $\lim_{t \uparrow \uparrow t_G} \barG(t) = \Delta G(t_G) > 0$.
\end{proof}

\subsection{Decomposition of locally absolutely continuous functions}
\label{sec:decomposition}

Let $F \lac G$. Define the functions $F^\uparrow, F^\downarrow,\vert F \vert:[0, \infty) \to \RR$ by
\begin{align*}
F^{\uparrow/\downarrow}(t)
&= 
\begin{cases}
\int_{(0, t]} \left( \frac{\diff F}{\diff G}(v) \right)^{+/-} \dd G(v) &\text{if } t < t_G, \\
0 &\text{if } t \geq t_G,
\end{cases}\\
\vert F \vert(t)
&= 
\begin{cases}
\int_{(0, t]} \left| \frac{\diff F}{\diff G}(v) \right| \dd G(v) &\text{if } t < t_G, \\
0 &\text{if } t \geq t_G.
\end{cases}
\end{align*}
$F^\uparrow, F^\downarrow, \vert F \vert$ are well defined by Lemma \ref{lem:ACloc:local density}, null at $0$, nonnegative and increasing on $[0, t_G)$, and satisfy
\begin{align}
\label{eqn:ACloc:decompostion:1}
F \1_{[0, t_G)}
&= F(0)\1_{[0, t_G)} + F^\uparrow- F^\downarrow,\\
\vert F \vert
&= F^\uparrow + F^\downarrow.\notag
\end{align}
Restricted to $[0,t_G)$, $\vert F \vert$ is simply the total variation of $F$ and $F^{\uparrow/\downarrow}$ is the positive/negative variation of $F$ shifted to null at $0$.

The following result shows that if $F \lac G$ and $\op{G} F \in L^1(\diff G)$, then the analogous properties hold for $F^\uparrow, F^\downarrow, \vert F \vert$, too.

\begin{lemma}
\label{lem:op:decomposition}
Let $F \lac G$ be such that $\op{G} F \in L^1(\diff G)$.
Then $F^\uparrow, F^\downarrow, \vert F \vert\lac G$ and $\op{G} (F^\uparrow)$, $\op{G}(F^\downarrow)$, $\op{G} \vert F \vert$ are in $L^1(\diff G)$. Moreover,
\begin{align}
\label{eqn:lem:op:decomposition:1}
\op{G} F
&= F(0) + \op{G} (F^\uparrow) - \op{G} (F^\downarrow)\;\; \diff G \text{-a.e.},\\
\label{eqn:lem:op:decomposition:2}
\op{G} \vert F \vert
&= \op{G} (F^\uparrow) + \op{G} (F^\downarrow) \;\; \diff G \text{-a.e.}
\end{align}
\end{lemma}

\begin{proof}
$F^\uparrow, F^\downarrow, \vert F \vert \lac G$ is clear from the definitions, and \eqref{eqn:lem:op:decomposition:1} and \eqref{eqn:lem:op:decomposition:2} are easy calculations. Among the remaining claims, we only show $\op{G} (F^\uparrow) \in L^1(\diff G)$; $\op{G} (F^\downarrow) \in L^1(\diff G)$ follows analogously, and then $\op{G} \vert F\vert \in L^1(\diff G)$ follows from \eqref{eqn:lem:op:decomposition:2}.

We show that $\op{G} (F^\uparrow) \in L^1(\diff G)$ by using the implication ``(b) $\Rightarrow$ (a)'' in Lemma \ref{lem:op:integrability}. On the one hand, nonnegativity of $F^\uparrow$ gives $\limsup_{t \uparrow\uparrow t_G} F^\uparrow(t)\barG(t) \geq 0 > - \infty$. On the other hand, \eqref{eqn:ACloc:decompostion:1} together with nonnegativity of $F^\uparrow$ gives
\begin{align*}
\op{G} (F^\uparrow)
&= F^\uparrow(\cdot -) - \left(\frac{\diff F}{\diff G} \right)^+\barG
\geq \left((F(\cdot-) - F(0)) - \frac{\diff F}{\diff G}\barG \right)\1_{\lbrace \frac{\diff F}{\diff G} > 0\rbrace } \notag\\
&\geq \op{G} F \1_{\lbrace \frac{\diff F}{\diff G}> 0\rbrace } - |F(0)| 
\geq -|\op{G} F| - |F(0)| \;\; \diff G \text{-a.e.~on } (0, t_G),
\end{align*}
and hence,
\begin{align*}
\int_{(0, t_G)} \left(\op{G} (F^\uparrow) (v)\right)^- \dd G(v)
&\leq \int_{(0, t_G)} \left|\op{G} F (v)\right| \dd G(v) + |F(0)|
< \infty.\qedhere
\end{align*}
\end{proof}

The following lemma is in some sense the counterpart to Lemma \ref{lem:op:integrability:DeltaG>0} for the case $\Delta G(t_G) = 0$.

\begin{lemma}
\label{lem:Fbar minus:integrability}
Let $F \lac G$ and suppose that $\Delta G(t_G) = 0$. The following are equivalent:
\begin{enumerate}
\item $F(\cdot -) \in L^1(\diff G)$ and $\op{G} F \in L^1(\diff G)$.
\item $F(\cdot -) \in L^1(\diff G)$ and $\left(\op{G} F \right)^- \in L^1(\diff G)$.
\item $F(\cdot -) \in L^1(\diff G)$ and $\left(\op{G} F \right)^+ \in L^1(\diff G)$.
\item $\frac{\diff F}{\diff G}\barG \in L^1(\diff G)$.
\item $\vert F \vert(\cdot-) \in L^1(\diff G)$.
\end{enumerate}
\end{lemma}

\begin{proof} ``(a) $\Rightarrow$ (b), (c)'': This is trivial.

``(b) $\Rightarrow$ (a)'': If $\op{G} F \not\in L^1(\diff G)$, then Lemma \ref{lem:op:integrability} implies $\lim_{t \uparrow\uparrow t_G} F(t)\barG(t) = -\infty$. Choose ${N \in \NN}$ large enough that $\int_{(0, t_G)} |F(v-)| \dd G(v) < N$, and $t \in [0, t_G)$ such that ${|F(v)\barG(v)| \geq N}$ for all $v \in [t, t_G)$.
Using $\Delta G(t_G) = 0$, this gives
\begin{align*}
\int_{(0, t_G)} |F(v-)| \dd G(v) 
&\geq \int_{(t, t_G)} \frac{|F(v-)\barG(v-)|}{\barG(v-)} \dd G(v) \\
&\geq N \int_{(t, t_G)} \frac{1}{\barG(t)} \dd G(v)
= N \frac{1 - G(t)}{\barG(t)} 
> \int_{(0, t_G)} |F(v-)| \dd G(v),
\end{align*}
which is a contradiction.

``(c) $\Rightarrow$ (a)'': This is analogous to the proof of ``(b) $\Rightarrow$ (a)''. 

``(a) $\Rightarrow$ (d)'': This follows immediately from the definition of $\op{G} F$.

``(d) $\Leftrightarrow$ (e)'': Using the definition of $\vert F \vert$, Fubini's theorem and $\Delta G(t_G) = 0$, we obtain
\begin{align*}
\int_{(0, t_G)} \vert F \vert(s-) \dd G(s)
&= \int_{(0, t_G)} \int_{(0, s)} \left| \frac{\diff F}{\diff G}(v) \right| \dd G(v) \dd G(s)
= \int_{(0, t_G)} \left| \frac{\diff F}{\diff G}(v) \right| \int_{(v, t_G)} \dd G(s) \dd G(v)\\
&= \int_{(0, t_G)} \left| \frac{\diff F}{\diff G}(v) \right| (1 - G(v)) \dd G(v).
\end{align*}
This immediately establishes both directions.

``(e) $\Rightarrow$ (a)'': On the one hand, (e) implies $F(\cdot-)\in L^1(\diff G)$ as $|F(v)| \leq \vert F(0) \vert + \vert F \vert(v)$ for $v \in [0, t_G)$, and on the other hand, (e) implies (d). Now the claim follows from the definition of $\op{G} F$.
\end{proof}

\begin{remark}
$F(\cdot-)\in L^1(\diff G)$ alone does not imply $\left(\op{G} F \right)^\pm \in L^1(\diff G)$. Indeed, let ${G: \RR \to [0, 1]}$ be given by $G(t) = t \1_{[0, 1)}(t) + \1_{[1, \infty)}(t)$, i.e., $\diff G$ is a uniform distribution on $(0, 1)$ with $t_G = 1$, and define $F:[0, \infty) \to \RR$ by $F(t) = \1_{[0, 1)}(t)\sin\frac{1}{1- t}$. Then an easy exercise in analysis shows
\begin{align*}
\int_{(0,1)} |F(v-)| \dd G(v)
&= \int_0^1 |F(v)| \dd v
< \infty, \\
\int_{(0,1)} \left(\op{G} F(v)\right)^\pm \dd G(v)
&= \int_0^1 \left( \sin\frac{1}{1 - v} - \frac{1}{1-v}\cos\frac{1}{1 - v}\right)^\pm \dd v
= \infty.
\end{align*}
\end{remark}

\section{Classification of single jump local martingales}
\label{sec:classification}

Let $(\Omega, \cF, P)$ be a probability space and $\gamma$ a fixed $(0, \infty)$-valued random variable with distribution function $G$. The filtration $\FF^\gamma = (\cF^\gamma_t)_{t\in[0,\infty]}$ given by
\begin{align}
\label{eqn:filtration}
\cF^\gamma_t
&= \sigma(\lbrace \gamma \leq s \rbrace : s \in (0,t])
\end{align}
is the smallest filtration with respect to which $\gamma$ is a stopping time. For any $F \lac G$, define the function $\zeta^F: [0,\infty] \times (0,\infty) \to \RR$ by
\begin{align}
\label{eqn:pf}
\zeta^F(t, v)
&= F( t) \1_{\lbrace t < v \rbrace } + \op{G} F(v) \1_{\lbrace t \geq v \rbrace },
\end{align}
where $\op{G}F:(0,\infty)\to\RR$ is defined by
\begin{align*}
\op{G} F(v)
&=
\begin{cases}
F(v-) - \frac{\diff F}{\diff G}(v) \barG(v), \quad &v \in (0, t_G),\\
F(t_G-) \1_{\lbrace \Delta G(t_G) > 0\rbrace }, &v \geq t_G, \text{ if } \lim_{t \uparrow \uparrow t_G} F(t) \text{ exists in }\RR, \\
0, &v \geq t_G, \text{ if } \lim_{t \uparrow \uparrow t_G} F(t) \text{ does not exist in }\RR;
\end{cases}
\end{align*}
cf.~Lemma \ref{lem:op}. Note that $\zeta^F$ is Borel-measurable and for each $t \in [0,\infty]$, $\zeta^F(t, \cdot)$ is unique up to $\diff G$-nullsets (because the local density $\frac{\diff F}{\diff G}$ is only $\diff G$-a.e.~unique on $(0,t_G)$). Now define the process $\mart{G}{F} = (\mart[t]{G}{F})_{t\in[0,\infty]}$ by
\begin{align}
\label{eqn:MGF}
\mart[t]{G}{F}
&= \zeta^F(t, \gamma)
= F(t) \1_{\lbrace t < \gamma \rbrace } + \op{G}F(\gamma) \1_{\lbrace t \geq \gamma \rbrace }.
\end{align}
$\mart{G}{F}$ is clearly $\FF^\gamma$-adapted and it is easy to see that modifying $\op{G} F$ on a $\diff G$-nullset leads to a process that is indistinguishable from the original process $\mart{G}{F}$. Every trajectory $\mart[\cdot]{G}{F} (\omega)$ is càdlàg and of finite variation on $[0, t_G)$, nonrandom until just before the random time $\gamma(\omega)$, and stays constant at $\op{G}F(\gamma(\omega))$ from time $\gamma(\omega)$ on. In particular,
\begin{align*}
\mart[t_G]{G}{F}
&= \op{G}F(\gamma) \;\; \as{P}
\end{align*}

The first line in the definition of $\op{G}F$ is chosen such that $\mart{G}F$ becomes an $\FF^\gamma$-martingale on the \emph{right-open} interval $[0,t_G)$. This result is well known in the literature (see e.g.~\cite{Dellacherie1970}). For the convenience of the reader, we provide a full proof here. In the following Sections \ref{sec:local martingales}--\ref{sec:H1 martingales}, we then classify the (local) martingale properties of $\mart{G}{F}$ when considered on the \emph{closed} interval $[0,t_G]$ (or, equivalently, on $[0,\infty]$).

\begin{lemma}
\label{lem:martingale}
The process $\mart{G}{F}$ is an $\FF^\gamma$-martingale on $[0, t_G)$.
\end{lemma}

\begin{proof}
For brevity, we set $M := \mart{G}{F}$. To check integrability, fix $0 \leq t < t_G$. Then the definition of $M$ and Lemma \ref{lem:op} give
\begin{align*}
\EX{\left\vert M_t\right\vert}
&\leq |F(t)| P[t < \gamma] + \EX{\left\vert\op{G} F(\gamma)\right\vert \1_{\lbrace t \geq \gamma\rbrace }} \\
&= |F(t)| (1-G(t)) + \int_{(0,t]} \left|\op{G} F(s)\right| \dd G(s) < \infty.
\end{align*}
To check the martingale property for $M$, fix $0 \leq s < t < t_G$. Then $t \geq \gamma$ on $\lbrace s \geq \gamma \rbrace $ gives
\begin{align*}
\cEX{M_t}{\cF^\gamma_s}
&= \cEX{M_t \1_{\lbrace s < \gamma \rbrace } + \op{G} F(\gamma) \1_{\lbrace s \geq \gamma\rbrace }}{\cF^\gamma_s} \\
&= \cEX{M_t \1_{\lbrace s < \gamma \rbrace }}{\cF^\gamma_s} + \op{G} F(\gamma) \1_{\lbrace s \geq \gamma\rbrace } \;\; \as{P}
\end{align*}
It is not hard to show that $\lbrace \gamma > s\rbrace $ is an atom of $\cF^\gamma_s$ (see e.g.~\cite{Dellacherie1970}, \cite{ChouMeyer1975} or \cite[Chapter IV, No.~104]{DellacherieMeyer1978}). Using this and \eqref{eqn:lem:op:integral formula} gives
\begin{align*}
\cEX{M_t \1_{\lbrace s < \gamma \rbrace }}{\cF^\gamma_s}
&= \cEX{M_t}{s < \gamma} \1_{\lbrace s < \gamma\rbrace } \\
&=\cEX{F(t) \1_{\lbrace t < \gamma \rbrace } + \op{G} F(\gamma) \1_{\lbrace s < \gamma \leq t\rbrace }}{s < \gamma} \1_{\lbrace s < \gamma\rbrace }\\
&=\frac{F(t)(1- G(t)) + \int_{(s, t]} \op{G} F(u) \dd G(u)}{1- G(s)} \1_{\lbrace s < \gamma\rbrace } \\
&=\frac{F(t)\barG(t) + \Big[-F(u)\barG(u)\Big]_s^t}{\barG(s)} \1_{\lbrace s < \gamma\rbrace } \\
&=\frac{F(s)\barG(s)}{\barG(s)} \1_{\lbrace s < \gamma\rbrace } = F(s) \1_{\lbrace s < \gamma\rbrace } \;\; \as{P}
\end{align*}
Thus, we may conclude that $\cEX{M_t}{\cF^\gamma_s} = F(s) \1_{\lbrace s < \gamma\rbrace } + \op{G} F(\gamma) \1_{\lbrace s \geq \gamma\rbrace } = M_s$ $\as{P}$
\end{proof}

We are now in a position to explain the structure of the function $\op{G}F$. On the one hand, if $\Delta G(t_G)=0$, then $\gamma < t_G$ $\as{P}$ and only the first line in the definition of $\op{G}F$ is relevant for $\mart{G}{F}$. On $(0,t_G)$, $\op{G}F$ is chosen such that $\mart{G}{F}$ becomes a martingale on the right-open interval $[0,t_G)$. On the other hand, if $\Delta G(t_G) > 0$, then $\gamma = t_G$ with positive probability. Assuming for the moment that $\mart{G}{F}$ is a martingale on $[0,t_G]$, the martingale convergence theorem implies that $\mart[t_G]{G}{F}=\lim_{t\uparrow\uparrow t_G} \mart[t]{G}{F}$ $\as{P}$ Evaluating the right-hand side on the event ${\lbrace \gamma = t_G \rbrace }$ yields $\mart[t_G]{G}{F} = \lim_{t \uparrow\uparrow t_G} F(t)=F(t_G-)$ on $\lbrace \gamma = t_G \rbrace $. This motivates the second line in the definition of $\op{G}F$. The last line is only relevant when $\Delta G(t_G) > 0$ and the left limit $F(t_G-)$ does not exist in $\RR$. But then $F$ must be of infinite variation on $[0,t_G]$ and
\begin{align*}
P[\mart[t]{G}{F}
&= F(t), t \in [0,t_G)]
\geq P[\gamma = t_G]
= \Delta G(t_G)
>0,
\end{align*}
so that $\mart{G}{F}$ fails to be a semimartingale on $[0,t_G]$ by Lemma \ref{lem:not semimartingale test}. Note that this is independent of the particular choice $\op{G}{F}(t_G):=0$.

\begin{remark}
\label{rem:credit risk}
Processes of the form $\mart{G}{F}$ for particular choices of $F$ play a special role in the modelling of credit risk, see e.g.~\cite{BieleckiRutkowski2002, JeanblancRutkowski2000, ElliottJeanblancYor2000}. We give two examples. We assume---as is usually done in the literature on credit risk---that $t_G = \infty$. First, for $F := \frac{1}{1 - G} = \frac{1}{\barG}$, we have ${\op{G} F= - \frac{\diff(F \barG)}{\diff G} = -\frac{\diff 1}{\diff G} = 0}$ and
\begin{align*}
\mart[t]{G}{F}
&= \frac{1}{1 - G(t)} \1_{\lbrace t < \gamma\rbrace } + 0 \cdot \1_{\lbrace t \geq \gamma\rbrace }
= \frac{1 - \1_{\lbrace t \geq \gamma\rbrace }}{1 - G(t)}, \quad t \in [0, \infty).
\end{align*}
This process is called $\hat M$ in \cite[Corollary 5.1]{JeanblancRutkowski2000}. Second, for
\begin{align*}
F(t)
&:= -\int_{(0, t]} \frac{\diff G(v)}{1-G(v-)}
= -\int_{(0, t]} \frac{\diff G(v)}{\barG(v-)}
\end{align*}
($=\log \barG(t)$ if $G$ is continuous), we have 
\begin{align*}
\op{G} F(v)
&= F(v-) - \frac{\diff F}{\diff G}(v)\barG(v)
= F(v) - \frac{\diff F}{\diff G}(v)\barG(v-)\\
&= F(v) + \frac{1}{\barG(v-)} \barG(v-)
= F(v) + 1 \;\; \diff G\text{-a.e.}
\end{align*}
and
\begin{align*}
\mart[t]{G}{F}
&= F(t) \1_{\lbrace t < \gamma\rbrace } + (F(\gamma) + 1)\1_{\lbrace t \geq \gamma\rbrace }
= \1_{\lbrace t \geq \gamma\rbrace } -\int_{(0, t\wedge \gamma]} \frac{\diff G(v)}{1- G(v-)}, \quad t \in [0, \infty).
\end{align*}
This process is called $M$ in \cite[Proposition 5.2]{JeanblancRutkowski2000}.

It is also often assumed that $G$ is absolutely continuous with respect to Lebesgue measure, i.e., $\diff G(t) = G'(t) \dd t$ for a nonnegative Borel-measurable function $G'$. In this case, the quantity $\kappa^G(t) := \frac{G'(t)}{\barG(t)}$ is the conditional probability density of the default time, given that default has not occurred up to time~$t$, and is often called \emph{hazard rate} or \emph{default intensity}. Clearly, any $F \lac G$ is also locally absolutely continuous with respect to Lebesgue measure, i.e., there is a Borel-measurable function $F'$ such that $\diff F(t) = F'(t) \dd t$. Now, $\mart{G}{F}$ has the following representation in terms of $F$, $F'$ and the hazard rate of $G$:
\begin{align*}
\mart[t]{G}{F}
&= F(t) \1_{\lbrace t<\gamma \rbrace } + \op{G} F(\gamma) \1_{\lbrace t \geq \gamma \rbrace }
= F(t) \1_{\lbrace t<\gamma \rbrace } + \left(F(\gamma-) - \frac{\diff F}{\diff G}(\gamma) \barG(\gamma) \right) \1_{\lbrace t \geq \gamma \rbrace },\\
&= F(t) \1_{\lbrace t<\gamma \rbrace } + \left(F(\gamma) - \frac{F'(\gamma)}{\kappa^G(\gamma)} \right) \1_{\lbrace t \geq \gamma \rbrace },\\
\end{align*}
or alternatively,
\begin{align*}
\mart[t]{G}{F}
&= F(t \wedge \gamma) - \frac{F'(\gamma)}{\kappa^G(\gamma)} \1_{\lbrace t \geq \gamma \rbrace }.
\end{align*}
\end{remark}

{\bfseries For the rest of this section (except for Section 3.4), we fix $F \lac G$ and set $M := \mart{G}{F}$ for brevity.}

The raw filtration generated by $M$, denoted by $\FF^M = (\cF^M_t)_{t\in[0,\infty]}$, is the smallest filtration such that $M$ is $\FF^M$-adapted. As $M$ is $\FF^\gamma$-adapted, $\FF^M$ is a subfiltration of $\FF^\gamma$. 

\begin{remark}
(a) While in the filtration $\FF^\gamma$, the value of $\gamma(\omega)$ is known at time $\gamma(\omega)$, this may not be true for the filtration $\FF^M$. In $\FF^M$, we can only tell the value of $\gamma(\omega)$ at time $\gamma(\omega)$ if we observe a jump of $M_\cdot(\omega)$ of a certain size at time $\gamma(\omega)$. However, if $\gamma(\omega) < t_G$ and $\frac{\diff F}{\diff G}(\gamma(\omega)) =0$, then $\op{G}F(\gamma(\omega))=F(\gamma(\omega)-)$ and $M_\cdot(\omega)$ has no jump at time $\gamma(\omega)$ (``$\gamma$ occurred, but we did not see it in the path of $M$''). A trivial example is given by $F \equiv 0$. Then $\op{G}F \equiv 0$, $M \equiv 0$, and $\FF^M$ contains no information about $\gamma$ at all.

(b) The filtrations $\FF^\gamma$ and $\FF^M$ need not be $P$-complete and $\FF^M$ need not be right-continuous in general ($\FF^\gamma$ is in fact right-continuous, see e.g.~\cite[Lemma II.3.24]{JacodShiryaev2003} or \cite[Chapter IV, No.~104]{DellacherieMeyer1978}). However, most of the results of martingale theory can be proved without these \emph{usual conditions}. In particular, the martingale convergence theorem and the convergence result for stochastic integrals stated in Lemma \ref{lem:not semimartingale test} do not rely on them.
\end{remark}

By the law of iterated expectations, if $M$ is an $\FF^\gamma$-martingale, then it is also an $\FF^M$-martingale. However, if $M$ is a local $\FF^\gamma$-martingale, then $M$ need not be a local $\FF^M$-martingale. The reason is that the $\FF^\gamma$-stopping times in the localising sequence need not be $\FF^M$-stopping times. To obtain stronger statements, we distinguish two filtrations in the definition of a local martingale. In particular, $M$ is called an \emph{$\FF^M$-local $\FF^\gamma$-martingale} if it is a local $\FF^\gamma$-martingale that admits a localising sequence consisting only of $\FF^M$-stopping times. We refer to Appendix \ref{sec:semimartingale theory} for the details and related (partly nonstandard) terminology for \hbox{(semi-)martingales}.

\subsection{Local martingale property on $[0, t_G]$}
\label{sec:local martingales}

The following preparatory lemma gives conditions for the integrability of $M$ on $[0,t_G]$.
\begin{lemma}
\label{lem:process:integrable}
The following are equivalent:
\begin{enumerate}
\item The process $M$ is integrable on $[0, t_G]$.
\item The random variable $M_{t_G}$ is integrable.
\item $\op{G} F \in L^1(\diff G)$.
\end{enumerate}
\end{lemma}

\begin{proof}
``(a) $\Rightarrow$ (b)'' is trivial, and ``(b) $\Rightarrow$ (a)'' holds because $M_t$ is integrable for $t\in[0,t_G)$ by Lemma \ref{lem:martingale}. \hbox{``(b) $\Leftrightarrow$ (c)''} follows from $M_{t_G} = \op{G}F(\gamma)$ $\as{P}$ and the fact that $\gamma$ has distribution function $G$ under $P$.
\end{proof}

\begin{theorem}
\label{thm:local martingale}
\mbox{}
\begin{enumerate}
\item If $\Delta G(t_G) = 0$, then $M$ is an $\FF^M$-local $\FF^\gamma$-martingale on $[0, t_G]$.
\item If $\Delta G(t_G) > 0$ and $M_{t_G}$ is {\bf not} integrable, then $M$ {\bf fails} to be a semimartingale on $[0, t_G]$.
\item If $\Delta G(t_G) > 0$ and $M_{t_G}$ is integrable, then $M$ is an $H^1$-$\FF^\gamma$-martingale on $[0, t_G]$.
\end{enumerate}
\end{theorem}

\begin{proof}
(a) We distinguish two cases for $F$. If there exists $t^* \in [0, t_G)$ such that $F(t) = F(t^*)$ for $t \in [t^*, t_G)$, then $\op{G}F(v) = F(v-) = F(t^*)$ for $\diff G$-a.e.~$v \in (t^*,t_G)$. Thus, $P$-a.e.~path of $M$ is constant on $[t^*, t_G]$. It follows that $M=M^{t^*}$ $\as{P}$, and so by Lemma \ref{lem:martingale}, $M$ is a uniformly integrable $\FF^\gamma$-martingale on $[0, t_G]$. If there is no such $t^*$, then there exists a strictly increasing sequence $(t_n)_{n \in \NN_0} \subset [0, t_G)$ such that
\begin{align*}
\lim_{n \to \infty} t_n
&= t_G \quad \text{and} \quad F(t_n) \neq F(t_{n-1}), \quad n \in \NN.
\end{align*}
For $n \in \NN$, define the random time $\tau_n: \Omega \to [0, t_G]$ by
\begin{align*}
\tau_n
&:= t_n \1_{\lbrace M_{t_n} - M_{t_{n-1}} \neq 0\rbrace } + t_G \1_{\lbrace M_{t_n} - M_{t_{n-1}}=0\rbrace }.
\end{align*}
Since $\lbrace M_{t_n} - M_{t_{n-1}} \neq 0\rbrace \in \cF^M_{t_n}$, $\tau_n$ is an $\FF^M$-stopping time, and
\begin{align*}
M_{t_n} - M_{t_{n-1}}
&= 
\begin{cases}
F(t_n) - F(t_{n-1}) \neq 0 &\text{if } \gamma > t_n, \\
\op{G} F(\gamma) - F(t_{n-1}) &\text{if } t_{n-1} < \gamma \leq t_n,\\
0 &\text{if } \gamma \leq t_{n-1},
\end{cases}
\end{align*}
so that
\begin{align*}
\lbrace \tau_n = t_G\rbrace = \lbrace M_{t_n} - M_{t_{n-1}}= 0\rbrace \subset \lbrace \gamma \leq t_n\rbrace \subset \lbrace M_{t_{n+1}} - M_{t_n}= 0\rbrace = \lbrace \tau_{n+1} = t_G\rbrace .
\end{align*}
This shows that the sequence $(\tau_n)_{n \in \NN}$ is increasing and satisfies
\begin{align}
\label{eqn:thm:local martingale:pf:30}
\lim_{n \to \infty} P[\tau_n = t_G]
&\geq \lim_{n \to \infty} P[\gamma \leq t_{n-1}]
= P[\gamma < t_G]
= 1;
\end{align}
here, we use the assumption $\Delta G(t_G) = 0$.
Moreover, for $n \in \NN$ and $s \in [0, t_G]$, it follows from the definition of $M$ that
\begin{align*}
M_s \1_{\lbrace \gamma \leq t_n\rbrace } &= F(s) \1_{\lbrace s < \gamma\rbrace } \1_{\lbrace \gamma \leq t_n\rbrace } + \op{G} F(\gamma) \1_{\lbrace s \geq \gamma \rbrace } \1_{\lbrace \gamma \leq t_n\rbrace } \\
&= F(t_n \wedge s) \1_{\lbrace t_n \wedge s < \gamma\rbrace } \1_{\lbrace \gamma \leq t_n\rbrace } + \op{G} F(\gamma) \1_{\lbrace \gamma \leq t_n \wedge s \rbrace } \1_{\lbrace \gamma \leq t_n\rbrace } \\
&= M_{t_n \wedge s} \1_{\lbrace \gamma \leq t_n\rbrace }.
\end{align*}
This together with $\lbrace \tau_n = t_G\rbrace \subset \lbrace \gamma \leq t_n\rbrace $ gives
\begin{align*}
M^{\tau_n}_s
&= 
M_s \1_{\lbrace \tau_n = t_G\rbrace } + M_{t_n \wedge s} \1_{\lbrace \tau_n = t_n\rbrace }
= M_{t_n \wedge s} \1_{\lbrace \gamma \leq t_n\rbrace } \1_{\lbrace \tau_n = t_G\rbrace } + M_{t_n \wedge s} \1_{\lbrace \tau_n = t_n\rbrace } \\
&= M_{t_n \wedge s} \1_{\lbrace \tau_n = t_G\rbrace } + M_{t_n \wedge s} \1_{\lbrace \tau_n = t_n\rbrace } 
= M^{t_n}_s.
\end{align*}
Now the claim follows from \eqref{eqn:thm:local martingale:pf:30} and the fact that by Lemma \ref{lem:martingale}, each $M^{t_n}$ is a uniformly integrable $\FF^\gamma$-martingale on $[0, t_G]$.

(b) Lemmas \ref{lem:process:integrable} and \ref{lem:op:integrability:DeltaG>0} show that $F$ is of infinite variation on $[0, t_G]$. Moreover, we have $P[M_t = F(t), t \in [0, t_G)] \geq P[\gamma = t_G] = \Delta G(t_G) > 0$. Now the claim follows from Lemma \ref{lem:not semimartingale test}.

(c) The assumption that $M_{t_G}$ is integrable together with Lemma \ref{lem:process:integrable} gives $\op{G}F \in L^1(\diff G)$. Since $\Delta G(t_G) > 0$, the limit $F(t_G-)=\lim_{t \uparrow \uparrow t_G} F(t)$ exists in $\RR$ by Lemma \ref{lem:op:integrability:DeltaG>0} and so there is $C > 0$ such that ${\sup_{s \in [0, t_G)} |F(s)| \leq C}$. Thus,
\begin{align}
\EX{\sup_{0 \leq s < t_G} \vert M_s \vert}
&= \EX{\sup_{0 \leq s < t_G} \left( |F(s)| \1_{\lbrace s < \gamma \rbrace } + \left|\op{G} F(\gamma )\right| \1_{\lbrace s \geq \gamma \rbrace } \right)}\notag \\
\label{eqn:thm:local martingale:pf:40}
&\leq C + \int_{(0, t_G]} \left| \op{G} F(v)\right| \dd G(v)
< \infty.
\end{align}
Moreover, using the definition of $\op{G}F$ and $\Delta G(t_G)>0$ in the third equality,
\begin{align*}
\lim_{t\uparrow\uparrow t_G} M_t
&= \lim_{t\uparrow\uparrow t_G} \left( F(t)\1_{\lbrace t<\gamma \rbrace } + \op{G}F(\gamma)\1_{\lbrace t \geq \gamma \rbrace } \right)
= F(t_G-)\1_{\lbrace \gamma = t_G \rbrace } + \op{G}F(\gamma)\1_{\lbrace \gamma < t_G \rbrace }\\
&= \op{G}F(\gamma)
= M_{t_G} \;\; \as{P}
\end{align*}
Since $M$ is an $\FF^\gamma$-martingale on $[0,t_G)$ by Lemma \ref{lem:martingale}, combining \eqref{eqn:thm:local martingale:pf:40} with the martingale convergence theorem shows that $M$ is an $\FF^\gamma$-martingale on the right-closed interval $[0,t_G]$.
\end{proof}

\subsection{Sub- and supermartingale property on $[0, t_G]$}
\label{sec:integrable local martingales}

\begin{definition}
\label{def:change in mass}
If $M$ is integrable on $[0, t_G]$, the \emph{change in mass of $M$} (on $[0, t_G]$) is defined as
\begin{align*}
\Delta \mu
&:= \EX{M_{t_G}} - \EX{M_0}
= \int_{(0, t_G]} \op{G} F(v) \dd G(v) - F(0).
\end{align*}
\end{definition}

If $M$ is integrable on $[0, t_G]$, it is a strict local martingale whenever $\Delta \mu \neq 0$. The following result gives a formula that allows to compute $\Delta \mu$ easily.

\begin{lemma}
\label{lem:mass}
Suppose that $M$ is integrable on $[0, t_G]$. Then 
\begin{align*}
\Delta \mu
&= -\lim_{t \uparrow\uparrow t_G} F(t)\barG(t) \1_{\lbrace \Delta G(t_G) = 0\rbrace },
\end{align*}
and for $0 \leq s < t_G$,
\begin{align}
\label{eqn:lem:mass}
\cEX{M_{t_G}}{\cF^\gamma_s} - M_s
&= \frac{\Delta \mu}{\barG(s)} \1_{\lbrace s < \gamma\rbrace } \;\; \as{P}
\end{align}
Moreover, $M$ is always an integrable $\FF^M$-local $\FF^\gamma$-martingale on $[0,t_G]$, and more precisely,
\begin{enumerate}
\item $M$ is an $\FF^\gamma$-submartingale and a strict local martingale on $[0, t_G]$ if and only if $\Delta \mu > 0$,
\item $M$ is an $\FF^\gamma$-martingale on $[0, t_G]$ if and only if $\Delta \mu = 0$,
\item $M$ is an $\FF^\gamma$-supermartingale and a strict local martingale on $[0, t_G]$ if and only if $\Delta \mu < 0$.
\end{enumerate}
\end{lemma}

\begin{proof}
If $\Delta G(t_G) > 0$, then $M$ is an $H^1$-$\FF^\gamma$-martingale on $[0, t_G]$ by Theorem \ref{thm:local martingale} (c), and all claims follow. So suppose that $\Delta G(t_G) = 0$. Then $M$ is an $\FF^M$-local $\FF^\gamma$-martingale on $[0,t_G]$ by Theorem \ref{thm:local martingale} (a). Moreover, using $\Delta G(t_G) = 0$ and \eqref{eqn:lem:op:integrability:open} gives
\begin{align*}
\Delta\mu
&=\EX{M_{t_G}} - \EX{M_0}
= \int_{(0, t_G]} \op{G} F(v) \dd G(v) - F(0)\\
&= \int_{(0, t_G)} \op{G} F(v) \dd G(v) - F(0)
= - \lim_{t \uparrow\uparrow t_G} F(t)\barG(t).
\end{align*}
To establish \eqref{eqn:lem:mass}, fix $0 \leq s < t_G $. Using $M_{t_G}=\op{G}F(\gamma)$ $\as{P}$, the fact that $\lbrace s < \gamma\rbrace $ is an atom of $\cF^\gamma_s$, $\Delta G(t_G)=0$ and \eqref{eqn:lem:op:integrability:open} gives
\begin{align*}
\cEX{M_{t_G}}{\cF^\gamma_s}
&= \cEX{\op{G} F(\gamma)}{s < \gamma} \1_{\lbrace s < \gamma\rbrace } + \op{G} F(\gamma) \1_{\lbrace s \geq \gamma\rbrace } \\
&= \frac{\int_{(s, t_G]} \op{G} F(u) \dd G(u)}{1- G(s)} \1_{\lbrace s < \gamma\rbrace } + \op{G} F(\gamma) \1_{\lbrace s \geq \gamma\rbrace } \\
&= \frac{\int_{(s, t_G)} \op{G} F(u) \dd G(u)}{1- G(s)} \1_{\lbrace s < \gamma\rbrace } + \op{G} F(\gamma) \1_{\lbrace s \geq \gamma\rbrace } \\
&= \frac{F(s)\barG(s)- \lim_{t \uparrow\uparrow t_G} F(t)\barG(t)}{\barG(s)} \1_{\lbrace s < \gamma\rbrace } + \op{G} F(\gamma) \1_{\lbrace s \geq \gamma\rbrace } \\
&=\frac{F(s)\barG(s) + \Delta \mu}{\barG(s)} \1_{\lbrace s < \gamma\rbrace } + \op{G} F(\gamma) \1_{\lbrace s \geq \gamma\rbrace } \\
&= M_s + \frac{\Delta \mu}{\barG(s)} \1_{\lbrace s < \gamma\rbrace } \;\; \as{P}
\end{align*}
The remaining claims are straightforward.
\end{proof}

The next result shows that if $M$ is integrable on $[0, t_G]$, then it can be naturally decomposed into its initial value $M_0$ and the difference of two supermartingales starting at $0$, i.e., it is a quasimartingale (cf.~\cite[Theorem VI.40]{DellacherieMeyer1982}).

\begin{corollary}
\label{cor:process:decomposition}
Let $M = \mart{G}{F}$ be integrable on $[0, t_G]$. Set $M^\uparrow := \mart{G}{(F^\uparrow)}$ and $M^\downarrow := \mart{G}{(F^\downarrow)}$. Then $M^\uparrow$ and $M^\downarrow$ are $\FF^\gamma$-supermartingales on $[0, t_G]$, start at $0$, and satisfy
\begin{align*}
M
&= M_0 + M^\uparrow - M^\downarrow \;\; \as{P}
\end{align*}
\end{corollary}

\begin{proof}
It follows from Lemmas \ref{lem:op:decomposition} and \ref{lem:process:integrable} that $M^\uparrow$ and $M^\downarrow$ are well defined, integrable on $[0, t_G]$ and start at $0$. Nonnegativity of $F^\uparrow, F^\downarrow$ and Lemma \ref{lem:mass} give the supermartingale property. The decomposition result follows from the definitions of $M, M^\uparrow$ and $M^\downarrow$, and from \eqref{eqn:ACloc:decompostion:1} and \eqref{eqn:lem:op:decomposition:1}.
\end{proof}

\subsection{$H^1$-martingale property on $[0, t_G]$}
\label{sec:H1 martingales}

If $M$ is integrable on $[0, t_G]$ and $\Delta G(t_G) > 0$, then $M$ is automatically an $H^1$-martingale on $[0, t_G]$ by Theorem \ref{thm:local martingale}. If $\Delta G(t_G) = 0$, however, the situation is more delicate.

\begin{lemma}
\label{lem:H1 martingale}
Suppose that $\Delta G(t_G) = 0$. Then the following are equivalent:
\begin{enumerate}
\item $M$ is an $H^1$-martingale on $[0, t_G]$ (in the sense that there exists a filtration $\FF=(\cF_t)_{t\in[0,\infty]}$ of $\cF$ such that $M$ is an $H^1$-$\FF$-martingale on $[0,t_G]$, cf.~Definition \ref{def:martingale}).
\item $M$ is an $H^1$-$\FF^\gamma$-martingale on $[0, t_G]$.
\item $F(\cdot-) \in L^1(\diff G)$ and $\op{G} F \in L^1(\diff G)$.
\end{enumerate}
\end{lemma}

\begin{proof}
``(b) $\Rightarrow$ (a)'': This is trivial.

``(a) $\Rightarrow$ (c)'': Lemma \ref{lem:process:integrable} yields $\op{G}F \in L^1(\diff G)$. Moreover, using the definitions of $\zeta^F$ and $M$ in \eqref{eqn:pf} and \eqref{eqn:MGF} and the fact that $\Delta G(t_G) =0$, we obtain
\begin{align*}
\int_{(0,t_G)} \vert F(v-) \vert\dd G(v)
&\leq \int_{(0,t_G)}\sup_{t<v} \vert F(t) \vert \dd G(v)
= \int_{(0,t_G)}\sup_{t<v} \vert \zeta^F(t,v) \vert\dd G(v)\\
&\leq \int_{(0,t_G)} \sup_{t \in [0,t_G]} \vert \zeta^F(t,v) \vert \dd G(v)
= \EX{\sup_{t \in [0,t_G]} \vert M_t \vert}
< \infty.
\end{align*}

``(c) $\Rightarrow$ (b)'': As $M$ is a local $\FF^\gamma$-martingale on $[0,t_G]$ by Theorem \ref{thm:local martingale} (a), it suffices to show that 
$\EX{\sup_{t \in [0,t_G]} \vert M_t \vert} < \infty$. First, note that $\vert F (t) \vert \leq \vert F(0) \vert + \vert F \vert (t) \leq \vert F(0)\vert + \vert F \vert(v-)$ for $0 \leq t<v<t_G$. Thus, for $t \in [0,t_G]$ and $v \in (0,t_G)$,
\begin{align*}
\vert \zeta^F(t,v) \vert
&= \vert F(t) \vert \1_{\lbrace t<v \rbrace } + \vert \op{G}F(v) \vert \1_{\lbrace t \geq v \rbrace }
\leq \vert F(0) \vert + \vert F \vert(v-) + \vert \op{G}F(v) \vert.
\end{align*}
Using this together with the definition of $M$ in \eqref{eqn:MGF} and the fact that $\Delta G(t_G) = 0$, we get
\begin{align*}
\EX{\sup_{t \in [0,t_G]} \vert M_t \vert}
&= \int_{(0,t_G)} \sup_{t \in [0,t_G]} \vert \zeta^F(t,v) \vert \dd G(v) \\
&\leq \vert F(0) \vert + \int_{(0,t_G)} \vert F \vert(v-) \dd G(v) + \int_{(0,t_G)} \vert \op{G}F(v) \vert \dd G(v)
< \infty.\qedhere
\end{align*}
\end{proof}

As a corollary, we obtain a criterion which allows us to construct (uniformly integrable) martingales that are not $H^1$-martingales. A concrete example is given in Example \ref{ex:not H1 martingale} below.

\begin{corollary}
\label{cor:not H1 martingale}
Suppose that $\Delta G(t_G) = 0$. Assume that $ \left(\op{G} F \right)^-$ or $\left(\op{G} F\right)^+$ belongs to $L^1(\diff G)$ and that $\lim_{t \uparrow\uparrow t_G} F(t)\barG(t)=0$, but that $F(\cdot-) \not\in L^1(\diff G)$. Then $M$ is an $\FF^\gamma$-martingale but {\bf not} an $H^1$-martingale on $[0, t_G]$.
\end{corollary}

\begin{proof}
Lemmas \ref{lem:op:integrability}, \ref{lem:process:integrable} and \ref{lem:mass} show that $M$ is an $\FF^\gamma$-martingale on $[0, t_G]$. That $M$ fails to be an $H^1$-martingale on $[0, t_G]$ follows from Lemma \ref{lem:H1 martingale}.
\end{proof}

\begin{remark}
Even if $M$ is an $H^1$-martingale on $[0, t_G]$, one may have $F \not\in L^1(\diff G)$. Indeed, define $G: \RR \to [0, 1]$ by $G(t) = \frac{1}{\e -1} \sum_{k = 1}^\infty \frac{1}{k!} \1_{[k, \infty)}(t)$ and $F: [0, \infty) \to \RR$ by $F(t) = {\sum_{k = 1}^\infty (k-1)! \,\1_{[k, k+1)}(t)}$. Then $t_G = \infty$ and for $n \in \NN$,
\begin{align*}
\int_{(0, \infty)} |F(v-)| \dd G(v)
&= \sum_{n = 1}^\infty F(n - 1) \Delta G(n) 
= \frac{1}{\e -1} \sum_{n = 2}^\infty \frac{1}{n(n-1)}
< \infty,\\
\int_{(0, \infty)} |F(v)| \dd G(v)
&= \sum_{n = 1}^\infty F(n) \Delta G(n)
= \frac{1}{\e -1} \sum_{n = 1}^\infty \frac{1}{n}
= \infty.
\end{align*}
So $F(\cdot-) \in L^1(\diff G)$ but $F \not\in L^1(\diff G)$. It remains to show that $M$ is an $H^1$-martingale. In view of Lemma \ref{lem:H1 martingale}, the fact that $F(\cdot-) \in L^1(\diff G)$ and the definition of $\op{G}F$, this boils down to proving that
\begin{align*}
\int_{(0, \infty)} \left| \frac{\diff F}{\diff G}(v)\barG(v) \right| \dd G(v)
&= \sum_{n = 1}^\infty \left|\frac{\Delta F (n)}{\Delta G(n)}\right|\barG(n) \Delta G(n)
= \sum_{n = 1}^\infty |\Delta F(n)| \barG(n)
\end{align*}
is finite. But this is true, because for $n \in \NN$,
\begin{align*}
|\Delta F(n)|
&= F(n) - F(n-1) \leq F(n)
= (n-1)!,\\
\barG(n)
&=1 - G(n)
= \frac{1}{\e -1} \sum_{k = {n+1}}^\infty \frac{1}{k!} 
\leq \frac{\e }{\e - 1} \frac{1}{(n +1)!}.
\end{align*}
\end{remark}

\subsection{Summary and examples}
\label{sec:martingale examples}

The flow chart in Figure \ref{fig:diagram} summarises the results of the previous sections. It gives the conditions one has to check in order to determine the (local) martingale properties of $\mart{G}{F}$. In this section, we give examples for four of the five cases one can end up with; examples for the fifth case that $\mart{G}{F}$ is an $H^1$-martingale are easy to find (take, e.g., $F\lac G$ bounded with $\op{G} F$ bounded).

\begin{figure}[!ht]
\tikzstyle{decision} = [diamond, draw, text width=6em, text centered]
\tikzstyle{block} = [rectangle, draw, text width=5.5em, text centered, rounded corners, minimum height = 3em]
\tikzstyle{line} = [draw, ->]

\begin{tikzpicture}[node distance = 4cm, auto]
    \node [decision] (01) {$\Delta G(t_G) = 0$};
    \node [decision, below of=01] (02) {$\op{G} F \;\diff G \text{-int.}$};
    \node [decision, below of=02] (03) {${F(t)\barG(t)\to 0}$\\as $t\uparrow\uparrow t_G$};
    \node [decision, below of=03] (04) {$F(\cdot-) \;\diff G \text{-int.}$};
    \node [decision, left of=02] (05) {$\op{G} F \;\diff G \text{-int.}$};
    \node [block, right of=02] (06) {nonint. loc.~mart.};
    \node [block, right of=03] (07) {int.~strict loc.~mart.};
    \node [block, right of=04] (08) {UI mart. not in $H^1$};
    \node [block, left of=05] (09) {not semimart.};
    \node [block, left of=04] (10) {$H^1$-mart.};
    \path [line] (01) -- node{yes} (02);
    \path [line] (02) -- node{yes} (03);
    \path [line] (03) -- node{yes} (04);
    \path [line] (02) -- node{no} (06);
    \path [line] (03) -- node{no} (07);
    \path [line] (04) -- node{no} (08);
    \path [line] (01) -| node[above, near start]{no} (05);
    \path [line] (05) -- node[above]{no} (09);
    \path [line] (05) -- node{yes} (10);
    \path [line] (04) -- node[above]{yes} (10);
\end{tikzpicture}
\caption{Decision diagram for single jump local martingales.}
\label{fig:diagram}
\end{figure}
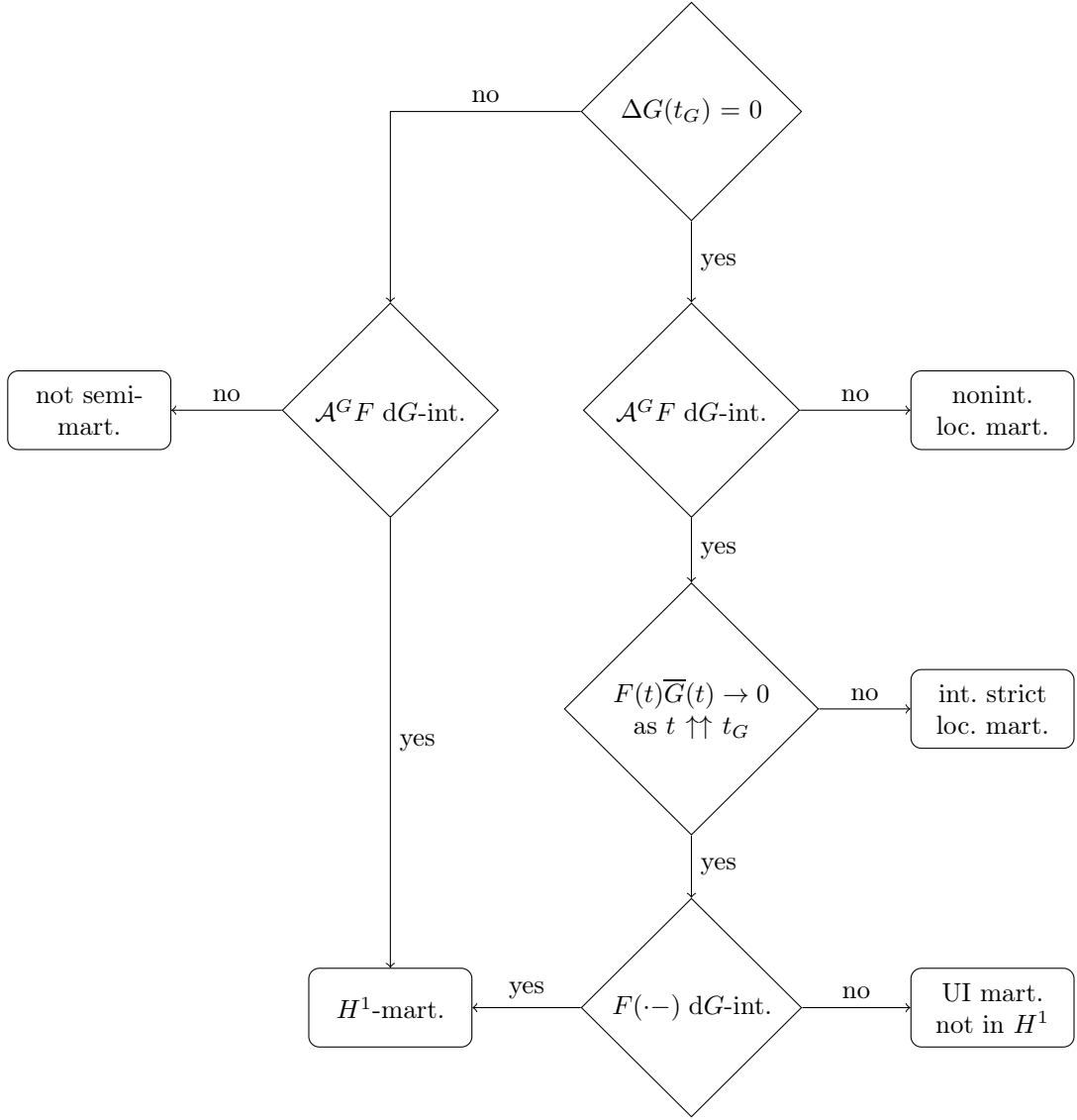

\begin{example}[A process which fails to be a semimartingale]
\label{ex:not semimartingale}
Let $G: \RR \to [0, 1]$ be given by $G(t) = \frac{t}{2} \1_{[0, 1)}(t) + \1_{[1, \infty)}(t)$, i.e., the law of the jump time $\gamma$ is a mixture of a uniform distribution on $(0, 1)$ and a Dirac measure at $1$. In particular, $t_G = 1$ and $\Delta G(t_G) = \frac{1}{2}$. The idea is to choose any $F \lac G$ that is of infinite variation on $[0,t_G]$. Then
\begin{align*}
P[M_t = F(t), t\in[0,t_G)]
&\geq P[\gamma = t_G]
= \Delta G(t_G)
= \frac{1}{2}
\end{align*}
by the definition of $\mart{G}{F}$, and Lemma \ref{lem:not semimartingale test} asserts that $\mart{G}{F}$ fails to be a semimartingale on $[0,t_G]$. (Alternatively, one can use Lemma \ref{lem:op:integrability:DeltaG>0} to infer that $\op{G}F \not\in L^1(\diff G)$ and then apply Lemma \ref{lem:process:integrable} and Theorem \ref{thm:local martingale} (b).) A concrete example is given by $F(t) = \1_{[0, 1)}(t)\sin\frac{1}{1-t}$, $t \geq 0$.
\end{example}

For the remaining examples, let $G: \RR \to [0, 1]$ be given by $G(t) = t \1_{[0, 1)}(t) + \1_{[1, \infty)}(t)$, i.e., $\gamma$ is uniformly distributed on $(0, 1)$. In particular, $t_G = 1$ and $\Delta G(t_G)=0$. Then for each $F \lac G$, $\mart{G}{F}$ is a local martingale on $[0, 1]$ by Theorem \ref{thm:local martingale} (a).

\begin{example}[A strict local martingale that fails to be integrable]
\label{ex:nonintegrable strict local martingale}
The idea is to find an $F \lac G$ such that $\mart[t_G]{G}{F}$ is not integrable or, equivalently by Lemma \ref{lem:process:integrable}, $\op{G}F \not\in L^1(\diff G)$. A concrete example is given by $F(t) = \1_{[0, 1)}(t)\sin\frac{1}{1- t}$, $t\geq 0$. Then $F \lac G$,
\begin{align*}
\op{G} F(v)
&= \sin\frac{1}{1-v} - \frac{1}{1-v} \cos\frac{1}{1-v} \;\; \diff G \text{-a.e.~on } (0, 1),
\end{align*}
and one can show that $\op{G}F\not\in L^1(\diff G)$. Indeed, it suffices to show that $\int_0^1 \left\vert \frac{1}{1-v} \cos\frac{1}{1-v} \right\vert \dd v = \infty$, or equivalently, $\int_1^\infty \left\vert \frac{\cos x}{x} \right\vert \dd x = \infty$. But for each $k \in \NN$,
\begin{align*}
\int_{(2k-1)\frac{\pi}{2}}^{(2k+1)\frac{\pi}{2}} \left\vert \frac{\cos x}{x} \right\vert \dd x
&\geq \frac{1}{(2k+1)\frac{\pi}{2}} \int_{(2k-1)\frac{\pi}{2}}^{(2k+1)\frac{\pi}{2}} \left\vert \cos x \right\vert \dd x
= \frac{2}{(2k+1)\frac{\pi}{2}}
\end{align*}
and summing over $k$ leads to an infinite series on the right-hand side.
\end{example}

\begin{example}[An integrable strict local martingale]
\label{ex:integrable strict local martingale}
Here the idea is to find $F \lac G$ with $F(0)>0$ and $\op{G}{F} = 0$ $\diff G$-a.e on $(0,1)$. This means that $\mart{G}{F}$ starts at $F(0)>0$ at time $0$ and ends up at zero at time $1$ $\as{P}$ Therefore, $\mart{G}{F}$ cannot be a martingale on $[0,1]$. A simple example is given by $F(t) = \frac{1}{1-t} \1_{[0, 1)}(t)$, $t\geq 0$. Then $F \lac G$,
\begin{align*}
\Delta\mu
&= -\lim_{t \uparrow\uparrow 1} F(t)\barG(t)
= -\lim_{t \uparrow\uparrow 1} \frac{1 - t}{1 - t}
= -1 \quad \text{and} \quad \op{G} F
= 0 \;\; \diff G \text{-a.e.~on } (0, 1).
\end{align*}
Thus, by Lemmas \ref{lem:process:integrable} and \ref{lem:mass}, $\mart{G}{F}$ is an integrable strict local martingale and a supermartingale.
\end{example}

\begin{example}[A martingale that fails to be in $H^1$]
\label{ex:not H1 martingale}
Finding $F \lac G$ such that $\mart{G}F$ is a martingale on $[0,1]$ that fails to be in $H^1$ is a bit tricky: if $F$ grows too slowly, then $\mart{G}{F}$ will be an $H^1$-martingale, but if $F$ grows too quickly, then $\mart{G}{F}$ will be a strict local martingale. The idea is to find an $F \lac G$ that satisfies the assumptions of Corollary \ref{cor:not H1 martingale}. Define $F:[0, \infty) \to \RR$ by $F(t) = \frac{1}{(1-t)\log \frac{\e}{1-t}} \1_{[0, 1)}(t)$.
Then $F \lac G$ is nonnegative and increasing on $[0, 1)$, and by monotone convergence,
\begin{align*}
\int_{(0, 1)} |F(v-)| \dd G(v) 
&= \lim_{t \uparrow\uparrow 1} \int_0^t F(v) \dd v 
= \lim_{t \uparrow\uparrow 1} \left[ \log\left(\log\frac{\e}{1-v}\right)\right]^t_0
= \lim_{t \uparrow\uparrow 1} \log \left( \log\frac{\e}{1-t}\right)
= \infty.
\end{align*}
Moreover,
\begin{align*}
\op{G} F(v)
&= \frac{1}{(1-v)\left(\log\frac{\e}{1-v}\right)^2} \;\; \diff G \text{-a.e.~on } (0, 1).
\end{align*}
Thus, $\op{G} F$ is nonnegative, and therefore $\left(\op{G} F \right)^- \in L^1(\diff G)$.
Finally,
\begin{align*}
-\Delta \mu
&= \lim_{t \uparrow\uparrow 1} F(t)\barG(t)
= \lim_{t \uparrow\uparrow 1} \frac{1}{\log\frac{\e}{1-t}}
= 0.
\end{align*}
Hence, by Corollary \ref{cor:not H1 martingale}, $\mart{G}{F}$ is a (uniformly integrable) martingale on $[0,1]$ but not in $H^1$.
\end{example}

\section{A counter-example in stochastic integration}
\label{sec:counterexample}

In this section, we consider the following problem from stochastic integration: Does there exist a pair $(M,H)$, where $M = (M_t)_{t\in[0,1]}$ is a (true) martingale and $H = (H_t)_{t\in[0,1]}$ an integrand with $0 \leq H \leq 1$ such that the stochastic integral $H \bullet M$ is a \emph{strict} local martingale? By the BDG inequality, $H \bullet M$ is again a martingale if $H$ is bounded and $M$ is an $H^1$-martingale. Nevertheless, the answer to the above question is positive as is shown in \cite[Corollaire VI.21]{Meyer1976} by an abstract existence proof using the Baire category theorem. It took, however, 30 years until a quite ingenious concrete example was published by Cherny \cite{Cherny2006}. He constructed the martingale integrator $M = (M_n)_{n \in \NN}$ recursively as follows. Starting with $M_0 = 1$, $M$ moves up or down at time $1$. If it moves down, it stays constant afterwards and if it moves up, $M$ can again move up or down at time $2$, and so on. The precise magnitudes of up and down 
movements are chosen such that $M$ becomes a uniformly integrable martingale that is not in $H^1$. Note that the structure of $M$ is precisely of the form $\mart{G}{F}$. Indeed, if $\gamma$ denotes the time when $M$ moves down and $G$ is the distribution function of $\gamma$, we can find a function $F \lac G$ such that the piecewise constant extension of $M$ to $[0,\infty)$ equals $\mart{G}{F}$. The integrand in Cherny's example is, up to a time transformation, the same as we use in Theorem \ref{thm:stochastic integral} below. The goal of this section is to provide an example of this kind, which works for $G$ the uniform distribution and any nondecreasing function $F$ such that $\mart{G}{F}$ is a uniformly integrable martingale but not in $H^1$.

In preparation of our counter-example, we first show that the class of single jump local martingales is closed under stochastic integration with bounded deterministic integrands.

\begin{proposition}
\label{prop:stochastic integral representation}
Let $F\lac G$ be such that $\mart{G}{F}$ is a local martingale on $[0,t_G]$. Moreover, let ${J:[0,\infty)\to\RR}$ be a bounded Borel-measurable function. Define ${F^J:[0,\infty) \to \RR}$ by ${F^J(t) = \int_{(0,t]} J(u)\dd F(u)}$ for $t\in[0,t_G)$ and $F^J(t)=0$ for $t \geq t_G$. Then $F^J \lac G$ and
\begin{align*}
J \bullet \mart{G}{F}
&= \mart{G}{F^J} \;\; \as{P}
\end{align*}
\end{proposition}

\begin{proof}
We only establish the result for the case $\Delta G(t_G) =0$, which corresponds to the setting of Theorem \ref{thm:stochastic integral} below.
Since $J$ is bounded and $F \lac G$, $J \in L^1_\loc(\diff F)$, so that $F^J$ is well defined. Clearly, $F^J \lac G$ with local density $\frac{\diff F^J}{\diff G} = J \frac{\diff F}{\diff G}$. Fix $t \in[0,t_G]$. Using the definition of $\mart{G}{F}$, on $\lbrace t < \gamma \rbrace $,
\begin{align}
J \bullet \mart{G}{F}_t
&= \int_{(0,t]} J(u) \dd F(u)
= F^J(t)
= \mart[t]{G}{F^J} \;\; \as{P}
\label{eqn:prop:stochastic integral representation:pf:10}
\end{align}
Using dominated convergence, on $\lbrace t \geq \gamma \rbrace \cap \lbrace \gamma < t_G \rbrace $,
\begin{align}
J \bullet \mart{G}{F}_t
&= \int_{(0,\gamma)} J(u) \dd \mart{G}{F}_u + J(\gamma)\Delta \mart{G}{F}_\gamma\notag\\
&= \int_{(0,\gamma)} J(u) \dd F(u) + J(\gamma)\left(\op{G}F(\gamma) - F(\gamma-)\right)\notag\\
&= F^J(\gamma-) - J(\gamma)\frac{\diff F}{\diff G}(\gamma)\barG(\gamma)
= F^J(\gamma-) - \frac{\diff F^J}{\diff G}(\gamma)\barG(\gamma)\notag\\
&= \op{G}F^J(\gamma)
= \mart[t]{G}{F^J} \;\; \as{P}
\label{eqn:prop:stochastic integral representation:pf:20}
\end{align}
As $\Delta G(t_G) = 0$, $\gamma < t_G$ $\as{P}$ and \eqref{eqn:prop:stochastic integral representation:pf:10} together with \eqref{eqn:prop:stochastic integral representation:pf:20} completes the proof.
\end{proof}

Throughout the rest of this section, let $G: \RR \to [0, 1]$ be given by $G(t) = t \1_{[0, 1)}(t) + \1_{[1, \infty)}(t)$, i.e., $\gamma$ is uniformly distributed on $(0, 1)$. In particular, $t_G = 1$ and $\Delta G(t_G)=0$. Moreover, we always consider the filtration $\FF^\gamma$ introduced in Section \ref{sec:classification}.

\begin{theorem}
\label{thm:stochastic integral}
Let $F \lac G$ be such that $\mart{G}{F}$ is a martingale on $[0, 1]$ which is not in $H^1$ (cf.~Corollary \ref{cor:not H1 martingale}). Suppose that $F$ is nondecreasing on $(0, 1)$ and satisfies $F(0) = 0$. Define the \emph{deterministic} integrand $J = (J_t)_{t \in [0,1]}$ by
\begin{align*}
J
&= \sum_{n = 0}^\infty \1_{\rdbrack 1-2^{-2n}, 1- 2^{-(2n+1)}\rdbrack}.
\end{align*}
Then the stochastic integral $J \bullet \mart{G}{F}$ is a strict local martingale on $[0,1]$.
\end{theorem}

A concrete example for a function $F$ satisfying all conditions of Theorem \ref{thm:stochastic integral} is given in Example \ref{ex:not H1 martingale}. The choice of $J$ is inspired by Cherny \cite{Cherny2006}.
\begin{proof}
Proposition \ref{prop:stochastic integral representation} yields $F^J\lac G$ and $J \bullet \mart{G}{F} = \mart{G}{F^J}$ $\as{P}$ By Theorem \ref{thm:local martingale}, $\mart{G}{F^J}$ is a local $\FF^\gamma$-martingale on $[0,1]$. It will be a strict local martingale if $\mart[1]{G}{F^J}$ is not integrable, which holds by Lemma \ref{lem:process:integrable} if and only if
\begin{align}
\label{eqn:thm:stochastic integral:pf:10}
\int_{(0,1)}| \op{G} F^J(t) | \dd G(t)
&= \int_0^1 \left| F^J(t) - J(t) \frac{\diff F}{\diff G}(t) (1- G(t)) \right| \dd t
= \infty.
\end{align}
Since $\mart{G}{F}$ is a martingale on $[0,1]$, $\op{G}F \in L^1(\diff G)$ by Lemma \ref{lem:process:integrable} and so
\begin{align*}
\int_0^1 \Big| J(t) F(t) - J(t) \frac{\diff F}{\diff G}(t) (1- G(t)) \Big| \dd t
&= \int_0^1 J(t) | \op{G} F(t) |\dd t \\
&\leq \int_0^1 | \op{G} F(t) | \dd t
= \int_{(0,1)} | \op{G} F(t) | \dd G(t)
< \infty.
\end{align*}
In order to establish \eqref{eqn:thm:stochastic integral:pf:10}, it thus suffices to show that
\begin{align}
\label{eqn:thm:stochastic integral:pf:20}
\int_0^1 \left| F^J(t) - J(t) F(t) \right| \dd t
&= \int_0^1 \left( (1- J(t)) F^J(t) + J(t) F^{1-J}(t) \right) \dd t
= \infty.
\end{align}
For $n \in \NN_0$, set $t_n := 1 - 2^{-n}$ and $t_{-1}:=-1$. We note that $t_{m + 1} - t_m = \frac{1}{2} (t_m - t_{m-1}) = \frac{1}{3}(t_{m +1} - t_{m -1})$ for $m \in \NN_0$ and that $F^J$ and $F^{1-J}$ are constant on $\lbrace J = 0 \rbrace $ and $\lbrace J = 1 \rbrace $, respectively. Using this, the nonnegativity of $F^{1-J}$ and $F$, the fact that $F = F^J + F^{1-J}$ is nondecreasing on $(0,1)$, we obtain
\begin{align}
&\int_0^1 \Big( (1- J(t)) F^J(t) + J(t) F^{1-J}(t) \Big) \dd t \notag\\
&= \sum_{n = 0}^\infty \left( F^J(t_{2n+1}) (t_{2n+2} - t_{2n +1}) + F^{1-J}(t_{2n+1}) (t_{2n+1} - t_{2n}) \right) \notag\\
&\geq \frac{1}{2} \sum_{n = 0}^\infty F(t_{2n+1}) (t_{2n+1} - t_{2n})
= \frac{1}{6} \sum_{n =0}^\infty F(t_{2n+1}) (t_{2n+1} - t_{2n -1})
\geq \frac{1}{6} \int_{(0,1)} F(t) \dd G(t).
\label{eqn:thm:stochastic integral:pf:30}
\end{align}
Since $\mart{G}{F}$ is not in $H^1$ by assumption, but $\mart{G}{F} \in L^1(\diff G)$, we get $F(\cdot -) \not\in L^1(\diff G)$ by Lemma \ref{lem:H1 martingale}. This implies $F\not\in L^1(\diff G)$ because $F$ is nondecreasing. Combining this with \eqref{eqn:thm:stochastic integral:pf:30} shows that \eqref{eqn:thm:stochastic integral:pf:20} holds true.
\end{proof}

\appendix

\section{Elements of real analysis}
\label{sec:analysis}

\begin{definition}
Let $T \in (0, \infty]$. A function $L: [0, \infty) \to \RR$ is called a \emph{left-continuous step function on $[0, T)$} if it is of the form
\begin{align*}
L
&= \sum_{j = 1}^k a_j \1_{(t_{j-1}, t_j]},
\end{align*}
where $k \in \NN$, $0 \leq t_0 < t_1 < \cdots < t_k < T$ and $a_j \in \RR$, $j = 1, \ldots, k$.
If $F: [0, T] \to \RR$ is any other function, we define the \emph{elementary integral of $L$ with respect to $F$ on $(0, T]$} by
\begin{align*}
\int_{(0, T]} L(t) \dd F(t)
&:= \sum_{j = 1}^k a_j \left(F(t_j) - F(t_{j-1}) \right).
\end{align*}
\end{definition}

The following result is an easy exercise in analysis.

\begin{lemma}
\label{lem:infinite variation}
Let $T \in (0, \infty]$ and $F: [0, T] \to \RR$ be a function which is of infinite variation on $[0, T]$. Then for each $n \in \NN$, there exists a left-continuous step function $L_n: [0, \infty) \to \RR$ on $[0, T)$ with $\sup_{t \geq 0} |L_n(t)| \leq 1/n$ and 
\begin{align*}
\int_{(0, T]} L_n(t) \dd F(t)
&\geq 1.
\end{align*}
\end{lemma}

\section{Elements of (semi-)martingale theory}
\label{sec:semimartingale theory}

Throughout this section, we fix a probability space $(\Omega, \cA, P)$ and a time horizon $T^*\in(0,\infty]$. Moreover, $I \subset [0,T^*]$ is assumed to be an interval of the form $[0,T)$ or $[0,T]$ for some $T \in (0,T^*]$.
\begin{definition}
\label{def:martingale}
Fix a stochastic process $X = (X_t)_{t \in [0,T^*]}$ and a filtration $\FF=(\cF_t)_{t \in [0,\infty]}$ of $\cA$.

\begin{enumerate}
\item $X$ is of \emph{finite variation on $I$} if for $P$-a.e.~$\omega$, the function $t \mapsto X_t(\omega)$ is of finite variation and càdlàg on $I$.
\item $X$ is called \emph{$\FF$-adapted on $I$} if $X_t$ is $\cF_t$-measurable for each $t \in I$.
\item $X$ is called \emph{integrable on $I$} if $\EX{\vert X_t \vert} < \infty$ for each $t \in I$.
\item $X$ is an \emph{$\FF$-(sub/super)martingale on $I$} if for $P$-a.e.~$\omega$, the function $t \mapsto X_t(\omega)$ is càdlàg on $I$ and $X$ is $\FF$-adapted on $I$, integrable on $I$, and satisfies the \emph{$\FF$-(sub/super)martingale property on $I$}, i.e.,
\begin{align*}
\cEX{X_t}{\cF_s}\;
&(\geq, \leq) = X_s \quad \text{for all } s \leq t \text{ in } I.
\end{align*}
$X$ is a \emph{(sub/super)martingale on $I$} if there exists a filtration $\FF'=(\cF'_t)_{t \in [0,\infty]}$ of $\cA$ such that $X$ is an $\FF'$-(sub/super)martingale on $I$.

\item $X$ is an \emph{$H^1$-$\FF$-martingale} on $I$ if $X$ is an $\FF$-martingale on $I$ and
\begin{align*}
\EX{\sup_{t \in I} \vert X_t \vert}
&< \infty.
\end{align*}
$X$ is an \emph{$H^1$-martingale} if there exists a filtration $\FF'=(\cF'_t)_{t \in [0,\infty]}$ of $\cA$ such that $X$ is an $H^1$-$\FF'$-martingale on $I$.

\item Let $\GG=(\cG_t)_{t \in [0,\infty]}$ be a subfiltration of $\FF$.
$X$ is a \emph{$\GG$-local $\FF$-martingale on $I$} if there exists an increasing sequence of $\GG$-stopping times $(\tau_n)_{n \in \NN}$ with values in $I \cup \lbrace T \rbrace $ such that for each $n \in \NN$, $X^{\tau_n}$ is an $\FF$-martingale on $I$, and
\begin{enumerate}
\item in case of $T \not\in I$, $\lim_{n \to \infty} \tau_n = T$ $\as{P}$,
\item in case of $T \in I$, $\lim_{n \to \infty} P[\tau_n = T] = 1$.
\end{enumerate}
In both cases, the sequence $(\tau_n)_{n \in \NN}$ is called a \emph{$\GG$-localising sequence} (for $X$). An $\FF$-local $\FF$-martingale on $I$ is simply called a \emph{local $\FF$-martingale on $I$}.
$X$ is a \emph{local martingale on $I$} if there exists a filtration $\FF'=(\cF'_t)_{t \in [0,\infty]}$ of $\cA$ such that $X$ is a local $\FF'$-martingale on $I$.
$X$ is a \emph{strict local martingale on $I$} if it is a local martingale on $I$, but not a martingale on $I$. 

\item $X$ is an \emph{$\FF$-semimartingale on $I$} if there are processes $M = (M_t)_{t\in I}$ and $A=(A_t)_{t\in I}$ such that $X = M +A$, where $M$ is local $\FF$-martingale on $I$ and $A$ is $\FF$-adapted on $I$ and of finite variation on $I$.
$X$ is a \emph{semimartingale on $I$} if there exists a filtration $\FF'=(\cF'_t)_{t \in [0,\infty]}$ of $\cA$ such that $X$ is an $\FF'$-semimartingale on $I$.
\end{enumerate}
Whenever we drop the qualifier ``on $I$'' in the above notations it is understood that $I = [0,T^*]$.
\end{definition}

The following result is a standard exercise in probability theory.
\begin{proposition}
Fix a stochastic process $X = (X_t)_{t \in [0,\infty]}$ and a filtration $(\cF_t)_{t \in [0,\infty]}$ of $\cA$.
\begin{enumerate}
\item Let $\FF'=(\cF'_t)_{t \in [0,\infty]}$ be a subfiltration of $\FF$ such that $X$ is $\FF'$-adapted on $I$. If $X$ is an $\FF$-(sub/super)martingale on $I$, then $X$ is also an $\FF'$-(sub/super)martingale on $I$.

\item Let $\GG=(\cG_t)_{t \in [0,\infty]}$, $\GG'=(\cG'_t)_{t \in [0,\infty]}$, $\FF'=(\cF'_t)_{t \in [0,\infty]}$ be subfiltrations of $\FF$ satisfying $\cG_t \subset \cG'_t \subset \cF'_t \subset \cF_t$ for each $t \in I$ and such that $X$ is $\FF'$-adapted on $I$. If $X$ is a $\GG$-local $\FF$-martingale on $I$, then it is also a $\GG'$-local $\FF'$-martingale on $I$.
\end{enumerate}
\end{proposition}

\begin{definition}
\label{def:elementary stochastic integral}
Let $T \in (0, \infty]$ and $\FF=(\cF_t)_{t \in [0, \infty]}$ be a filtration of $\cA$. A stochastic process ${L = (L_t)_{t \geq 0}}$ is called an \emph{$\FF$-elementary process on $[0, T)$} if there exist $\FF$-stopping times $\tau_0, \ldots, \tau_n$ with $0 \leq \tau_0 \leq \tau_1 \leq \cdots \leq \tau_n < T$ and bounded random variables $A_1, \ldots, A_n$ with each $A_j$ being $\cF_{\tau_{j-1}}$-measurable such that 
\begin{align*}
L
&= \sum_{j = 1}^k A_j \1_{\rdbrack \tau_{j-1}, \tau_j \rdbrack}.
\end{align*}
\end{definition}

Note that for each $\omega \in \Omega$, the path $L_{\cdot}(\omega)$ is a left-continuous step function on $[0, T)$.

\begin{definition}
Let $T \in (0, \infty)$. Let $X = (X_t)_{t \in [0, \infty]}$ be a stochastic process,
$\FF=(\cF_t)_{t \in [0, \infty]}$ a filtration of $\cA$ with respect to which $X$ is adapted on $[0, T]$ and $L = (L_t)_{t \geq 0}$ an $\FF$-elementary process on $[0, T)$. Define the \emph{$\FF$-elementary stochastic integral of $L$ with respect to $X$ on $(0, T]$} by
\begin{align*}
\left(\int_{(0, T]} L_t \dd X_t\right)(\omega)
&:= \int_{(0, T]} L_t(\omega) \dd X_t(\omega).
\end{align*}
\end{definition}

\begin{lemma}
\label{lem:semimartingale test}
Let $T \in (0, \infty)$. Let $X = (X_t)_{t \in [0, \infty]}$ be a stochastic process and $\FF=(\cF_t)_{t \in [0, \infty]}$ a filtration of $\cA$. If $X$ is an $\FF$-semimartingale on $[0, T]$, then for all $\FF$-elementary processes $(L^n_t)_{t \in [0, \infty)}$ on $[0, T)$ satisfying ${\lim_{n \to \infty} \sup_{t \geq 0, \omega \in \Omega} |L^n_t(\omega)| = 0}$,
\begin{align*}
P \text{-} \lim_{n \to \infty} \int_{(0, T]} L^n_t \dd X_t
&= 0.
\end{align*}
\end{lemma}

\begin{proof}
This follows immediately from \cite[Proposition 7.1.7]{WeizsaeckerWinkler1990} which is stronger than our result. Note that \cite{WeizsaeckerWinkler1990} work with \emph{general} filtrations which need not satisfy the usual conditions.
\end{proof}

\begin{lemma}
\label{lem:not semimartingale test}
Let $X = (X_t)_{t \in [0, \infty]}$ be a right-continuous stochastic process and $T \in (0, \infty)$. Suppose that there exists a deterministic function $F :[0, T] \to \RR$ such that $F$ is of infinite variation on $[0, T]$ and $P[X(t) = F(t), t \in [0, T)] =: \epsilon > 0$. Then $X$ is {\bf not} a semimartingale on $[0, T]$.
\end{lemma}

\begin{proof}
Seeking a contradiction, suppose there exists a filtration $\FF=(\cF_t)_{t \in [0, \infty]}$ of $\cA$ with respect to which $X$ is a semimartingale on $[0, T]$. Since $F$ is of infinite variation on $[0, T]$, by Lemma \ref{lem:infinite variation}, for each $n \in \NN$ there exists a left-continuous step function $L_n : [0, \infty) \to \RR$ on $[0, T)$ with $\sup_{t \geq 0} |L_n(t)| \leq 1/n$ and 
\begin{align*}
\int_{(0, T]} L_n(t) \dd F(t) \geq 1.
\end{align*}
For each $n \in \NN$, define the $\FF$-elementary process $(L^n_t)_{t \in [0, \infty)}$ on $[0, T)$ by $L^n_t(\omega) := L_n(t)$. Then $\lim_{n \to \infty} \sup_{t \geq 0, \omega \in \Omega} |L^n_t(\omega)| = 0$, but
\begin{align*}
P\left[\int_{(0, T]} L^n_t \dd X_t \geq 1\right]
&\geq P\left[\int_{(0, T]} L_n(t) \dd F(t) \geq 1, X(t) = F(t), t \in [0, T) \right]\\
&= P\left[ X(t) = F(t), t \in [0, T) \right]
= \epsilon \quad \text{for all } n \in \NN.
\end{align*}
This implies in particular that $(\int_{(0, T]} L_n(t) \dd X_t)_{n \in \NN}$ does not converge to $0$ in probability. Hence, $X$ fails to be an $\FF$-semimartingale on $[0,T]$ by Lemma \ref{lem:semimartingale test}, and we arrive at a contradiction.
\end{proof}


\providecommand{\bysame}{\leavevmode\hbox to3em{\hrulefill}\thinspace}
\providecommand{\MR}{\relax\ifhmode\unskip\space\fi MR }
\providecommand{\MRhref}[2]{%
  \href{http://www.ams.org/mathscinet-getitem?mr=#1}{#2}
}
\providecommand{\href}[2]{#2}

\end{document}